\RequirePackage{fix-cm}
\documentclass[smallextended]{svjour3}       % onecolumn (second format)
\smartqed  % flush right qed marks, e.g. at end of proof
\usepackage{graphicx}
\usepackage{placeins}
\usepackage{graphicx}%
\usepackage{multirow}%
\usepackage{amsmath,amssymb,amsfonts}%
\usepackage{amsmath}%
\usepackage{mathrsfs}%
\usepackage[title]{appendix}%
\usepackage{xcolor}%
\usepackage{textcomp}%
\usepackage{manyfoot}%
\usepackage{booktabs}%
\usepackage{algorithm}%
\usepackage{algorithmicx}%
\usepackage{algpseudocode}%
\usepackage{listings}%
\DeclareMathOperator*{\argmin}{argmin}
\usepackage{xparse}
\DeclareFontFamily{U}{ntxmia}{}
\DeclareFontShape{U}{ntxmia}{m}{it}{<-> ntxmia }{}
\DeclareFontShape{U}{ntxmia}{b}{it}{<-> ntxbmia }{}
\DeclareSymbolFont{lettersA}{U}{ntxmia}{m}{it}
\SetSymbolFont{lettersA}{bold}{U}{ntxmia}{b}{it}

\ExplSyntaxOn
\NewDocumentCommand{\varmathbb}{m}
 {
  \tl_map_inline:nn { #1 }
   {
    \use:c { varbb##1 }
   }
 }
\tl_map_inline:nn { ABCDEFGHIJKLMNOPQRSTUVWXYZ }
 {
  \exp_args:Nc \DeclareMathSymbol{varbb#1}{\mathord}{lettersA}{\int_eval:n { `#1+67 }}
 }
\exp_args:Nc \DeclareMathSymbol{varbbk}{\mathord}{lettersA}{169}
\ExplSyntaxOff
\newcommand{\mc}[1]{\mathcal{#1}}
\newcommand{\prox}{\mathrm{prox}}
\newcommand{\T}{\varmathbb{T}}
\newtheorem{fact}{Fact}
\usepackage[left=1in,right=1in]{geometry}
\begin{document}

\title{Optimal Acceleration for Proximal Minimization of the \\
Sum of Convex and Strongly Convex Functions
%\thanks{Grants or other notes
%about the article that should go on the front page should be
%placed here. General acknowledgments should be placed at the end of the article.}
}
% \subtitle{Do you have a subtitle?\\ If so, write it here}

% \titlerunning{Fast Douglas--Rachford Splitting}        % if too long for running head

\author{Govind M. Chari \and
        Uijeong Jang  \and
        Ernest K. Ryu \and
        Beh\c{c}et A\c{c}\i{}kme\c{s}e
}

\authorrunning{Chari, Jang, Ryu, and A\c{c}\i{}kme\c{s}e} % if too long for running head

\institute{Govind M. Chari \at
              University of Washington, Seattle \\
              \email{gchari@uw.edu}           %  \\
        \and
           Uijeong Jang \at
              University of California, Los Angeles \\
              \email{uijeongjang@math.ucla.edu}           %  \\
        \and
           Ernest K. Ryu \at
              University of California, Los Angeles \\
              \email{eryu@math.ucla.edu}           %  \\
        \and
           Beh\c{c}et A\c{c}\i{}kme\c{s}e \at
              University of California, Berkeley \\
              \email{behcet@berkeley.edu}           %  \\
}

\date{Received: date / Accepted: date}
% The correct dates will be entered by the editor

\maketitle

\begin{abstract}
When minimizing the sum of a convex and a strongly convex function, or when finding the zero of the sum of a monotone operator and a strongly monotone operator, Chambolle and Pock (2010) and Davis and Yin (2015) proposed accelerated mechanisms that achieve an $\mathcal{O}(1/N^2)$ convergence rate for the squared distance to the solution, but the optimality of this rate was not established. In this work, we present Fast Douglas--Rachford Splitting (FDR), an accelerated method that improves the constants established in the prior works, and provide a complexity lower bound establishing that both the $\mathcal{O}(1/N^2)$ convergence rate and the leading-order constant of FDR's rate are optimal.
% \keywords{Douglas--Rachford splitting \and acceleration \and complexity lower bounds}
% \PACS{PACS code1 \and PACS code2 \and more}
% \subclass{90C25 \and 47H05 \and 65K05}
\end{abstract}

\section{Introduction}
We consider the composite minimization problem

\begin{equation}\label{eq:problem}
    \underset{x \in \mathbb{R}^d}{\text{minimize}}
    \quad f(x) + g(x)
\end{equation}
where $f\colon \mathbb{R}^d \to \mathbb{R} \cup \{\infty\}$ is closed, convex, and proper and  $g\colon \mathbb{R}^d \to \mathbb{R} \cup \{\infty\}$ is closed, proper, and $\mu$-strongly convex. When $f$ and $g$ are \emph{proximable}, i.e., their proximal operators can be evaluated efficiently, Douglas--Rachford splitting (DRS) \cite{douglas1956numerical,lions1979splitting}
\begin{align*}
    x_{k+1/2}&=\prox_{\alpha g}(z_k)\\
    x_{k+1}&=\prox_{\alpha f}(2x_{k+1/2}-z_k)\\
    z_{k+1}&=z_k + x_{k+1}-x_{k+1/2},
\end{align*}
with stepsize $\alpha>0$, is a widely used solution method.

While convergence of DRS only requires $g$ to be convex, strong convexity of $g$ produces the rate

\[
\|x_N-x_\star\|^2=\mc{O}(1/N).
\]

The problem setup can be further generalized to the monotone inclusion problem
\begin{equation}\label{eq:mi}
    \underset{x \in \mathbb{R}^d}{\text{find}}
    \quad 0\in (\varmathbb{A}+\varmathbb{B})x
\end{equation}
where $\varmathbb{A}$ is maximal monotone and $\varmathbb{B}$ is maximal monotone and $\mu$-strongly monotone.

For a maximal monotone operator $\varmathbb{A}$, let
\[
    \varmathbb{J}_{\alpha \varmathbb{A}}= (\varmathbb{I}+\alpha \varmathbb{A})^{-1}
\]
denote its resolvent. Douglas--Rachford splitting applied to
\eqref{eq:mi}, with stepsize \(\alpha>0\), is
\begin{align*}
    x_{k+1/2} &= \varmathbb{J}_{\alpha \varmathbb{B}}(z_k)\\
    x_{k+1}   &= \varmathbb{J}_{\alpha \varmathbb{A}}(2x_{k+1/2}-z_k)\\
    z_{k+1}   &= z_k+x_{k+1}-x_{k+1/2},
\end{align*}
and attains the same $\mathcal{O}(1/N)$ rate under $\mu$-strong monotonicity of $\varmathbb{B}$.

In optimization, an \emph{acceleration} is a modification of a base algorithm that improves the worst-case convergence rate of a given performance measure. While the classical Nesterov acceleration \cite{nesterov1983method} achieves the optimal accelerated rate for function-value suboptimality, the modern theory of accelerations considers a broader range of performance measures, including distance to the solution and the squared fixed-point residual. In particular, the Chambolle--Pock~\cite{Chambolle2010} and Davis--Yin splitting~\cite{davis2017three} methods admit accelerated variants---based on mechanisms distinct from Nesterov's acceleration---that can be applied to Problem~\eqref{eq:problem} to obtain an accelerated rate of $\mc{O}(1/N^2)$ for the squared distance to the optimal solution. However, it was unknown whether this $\mc{O}(1/N^2)$ rate is optimal for this problem class, or whether the leading constants achieved by these accelerated methods are tight.

\paragraph{Contributions.}
This work presents two main contributions. First, we present an improved accelerated algorithm for Problems~\eqref{eq:problem} and~\eqref{eq:mi} whose leading constant improves upon those of the accelerated Chambolle--Pock and Davis--Yin methods by a factor of \(4\). Specifically, the iterates satisfy
\[
\|x_N-x_\star\|^2
\leq
\frac{\|x_0-x_\star\|^2+\|u_0-u_\star\|^2}{1+4N^2\mu^2}
\sim
\frac{\|x_0-x_\star\|^2+\|u_0-u_\star\|^2}{{\color{red}4}N^2\mu^2},
\]
where \(x_\star\) denotes an optimal solution, \(u_\star\) is a corresponding dual solution, $x_0$ and $u_0$ are the primal and dual initializations respectively, \(\mu\) is the strong convexity parameter, and \(N\) is the iteration count. Second, we establish a matching complexity lower bound showing that both the \(\mc{O}(1/N^2)\) convergence rate and the leading-order term \((\|x_0-x_\star\|^2+\|u_0-u_\star\|^2)/4N^2\mu^2\) are optimal for this problem class.

\subsection{Notation and preliminaries}

Throughout, $\mathbb{N}$ denotes the set of natural numbers, $\mathbb{R}^d$ the $d$-dimensional Euclidean space, $\langle \cdot, \cdot \rangle$ the standard inner product, and $\|\cdot\|$ the induced norm.

We recall standard definitions from convex analysis and monotone operator theory \cite{rockafellar1970convex,Bauschke2011,ryu2022LSCMO}.

A function $f\colon \mathbb{R}^d \to \mathbb{R} \cup \{\infty\}$ is \textit{closed} if its epigraph, defined as 

\[
\mathrm{epi} \, f = \{(x, t) \in \mathbb{R}^d \times \mathbb{R} \; | \; f(x) \le t\},
\]

is a closed set. It is \textit{convex} if

\[
f(\theta x + (1-\theta)y) \le \theta f(x) + (1-\theta)f(y) \quad \forall x, y \in \mathbb{R}^d \text{ and } \theta\in[0,1],
\]

and \textit{proper} if it never takes the value $-\infty$ and has finite value for some $x_0 \in \mathbb{R}^d$.

The \textit{subdifferential} of $f$ at $x$ is

\[
\partial f(x) = \{v \in \mathbb{R}^d \; | \; f(y) \ge f(x) + \langle v, y-x \rangle, \forall y \in \mathbb{R}^d\}.
\]

We write $f'(x) \in \partial f(x)$ to denote a subgradient.

A function $g\colon \mathbb{R}^d \to \mathbb{R} \cup \{\infty\}$ is $\mu$-strongly convex if $g(x) - \frac{\mu}{2}\|x\|^2$ is convex, or equivalently, if for all $x, y\in \mathbb{R}^d$ and all $v \in \partial g(x)$,

\[
g(y) \ge g(x) + \langle v, y-x \rangle + \frac{\mu}{2}\|y-x\|^2.
\]

If $f$ is closed, convex, and proper, then its subdifferential $\partial f$ is a \textit{monotone operator}, which means that

\[
\langle v - w, x-y \rangle \ge 0 
\quad \forall x,y \in \mathbb{R}^d,
\; \forall v \in \partial f(x),
\; \forall w \in \partial f(y).
\]

Similarly, if $g$ is closed, proper, and $\mu$-strongly convex, then $\partial g$ is a \textit{$\mu$-strongly monotone operator}, meaning that

\[
\langle v - w, x-y \rangle \ge \mu \|x-y\|^2 
\quad \forall x,y \in \mathbb{R}^d,
\; \forall v \in \partial g(x),
\; \forall w \in \partial g(y).
\]

For a function $f: \mathbb{R}^d \to \mathbb{R} \cup \{\infty\}$, define

\[
\underset{x\in \mathbb{R}^d}{\argmin} \; f(x) = \{z \in \mathbb{R}^d \; | \; f(z) \le f(x) \quad \forall x \in \mathbb{R}^d \}.
\]

For a closed, convex, and proper function $f$ and $\gamma > 0$, the proximal operator $\prox_{\gamma f}:\mathbb{R}^d \to \mathbb{R}^d$ is

\[
    \prox_{\gamma f}(x) = \underset{z\in \mathbb{R}^d}{\mathrm{argmin}} \; \left\{ f(z) + \frac{1}{2\gamma}\|z-x\|^2\right\}
\]

From the optimality conditions of this minimization problem, if $y = \prox_{\gamma f}(x)$, then

\[
0 \in \partial f(y) + \frac{1}{\gamma}(y-x).
\]

Equivalently, 

\[
y = x - \gamma f'(y).
\]

Consider the primal-dual problem pair

\begin{equation}\label{eq:primal}
    \underset{x \in \mathbb{R}^d}{\text{minimize}}
    \quad f(x) + g(x) \tag{\ensuremath{\mc{P}}}
\end{equation}

\begin{equation}\label{eq:dual}
    \underset{u \in \mathbb{R}^d}{\text{maximize}}
    \quad -f^*(-u) - g^*(u) \tag{\ensuremath{\mc{D}}}
\end{equation}
where $f\colon \mathbb{R}^d \to \mathbb{R} \cup \{\infty\}$ is closed, convex, and proper, and $g: \mathbb{R}^d \to \mathbb{R} \cup \{\infty\}$ is closed, proper, and $\mu$-strongly convex. Assume that the primal problem \eqref{eq:primal} admits a minimizer $x_\star$, which is unique by strong convexity of $g$, and total duality holds. Let $u_\star \in \partial g(x_\star)$ denote a corresponding dual solution to \eqref{eq:dual}. Then necessarily $-u_\star \in \partial f(x_\star)$ \cite{ryu2022LSCMO}.

\begin{lemma}\label{lem:fixed-point}
    If $(x_\star,u_\star)$ is a primal-dual solution (i.e. $u_\star \in \partial g(x_\star)$ and $-u_\star \in \partial f(x_\star)$), then $x_\star = \prox_{\gamma g}(x_\star + \gamma u_\star)$ and $x_\star = \prox_{\gamma f}(x_\star - \gamma u_\star)$ for $\gamma > 0$.
\end{lemma}

\begin{proof}
    We can verify the equalities by writing optimality conditions for the proximal operator

    \begin{align*}
    0 
    &\in \partial g(x_\star) 
       + \frac{1}{\gamma}\big(x_\star - (x_\star + \gamma u_\star)\big) \\
    &= \partial g(x_\star) - u_\star.
    \end{align*}
    
    Since $u_\star \in \partial g(x_\star)$, the inclusion holds.  

    \begin{align*}
    0 
    &\in \partial f(x_\star) 
       + \frac{1}{\gamma}\big(x_\star - (x_\star - \gamma u_\star)\big) \\
    &= \partial f(x_\star) + u_\star.
    \end{align*}
    
    Since $-u_\star \in \partial f(x_\star)$, the inclusion holds.  

\end{proof}

\subsection{Acceleration mechanisms and performance measures}
We quickly review other acceleration mechanisms that have been published in prior works. Note that each of these acceleration mechanisms is distinct in the sense that the metric they accelerate differs.

\paragraph{Nesterov-type Acceleration.}
The classical Nesterov accelerated gradient method (NAG) reduces the function-value suboptimality at an accelerated $\mathcal{O}(1/N^2)$ rate \cite{nesterov1983method}. In the modern literature, the Optimized Gradient Method (OGM) \cite{drori2014performance,kim2016optimized,taylor2017smooth} improves the leading constant of NAG by a factor of 2 and attains the exact optimal worst-case rate for this setup with an exact matching complexity lower bound \cite{drori2017exact}.

In the composite minimization setup where one function is smooth and the other is proximable, NAG can be extended to FISTA \cite{beck2009fast} and OGM to OptISTA \cite{jang2025computer}. Analogously, OptISTA is exactly optimal for the smooth composite convex minimization setup and improves the leading constant of FISTA by a factor of $2$.

In Problem~\eqref{eq:problem}, however, neither function is assumed to be smooth, so Nesterov-type acceleration does not apply directly. To clarify the distinction, suppose instead that \(g\) were \(L\)-smooth rather than \(\mu\)-strongly convex. Then FISTA applied to Problem \eqref{eq:problem} is

\begin{equation}
    \begin{aligned} 
    & x_{k+1} = \prox_{f/L}\left(y_k - \frac{1}{L} \nabla g(y_k)\right) \\
    & y_{k+1} = x_{k+1} + \left( \frac{t_k - 1}{t_{k+1}}\right) (x_{k+1}-x_{k}) \\
    \end{aligned}
    \tag{\textup{FISTA}}\label{eq:FISTA_Alg}
\end{equation}
where

\[
 y_1 = x_1 \in \mathbb{R}^d, \; t_1 = 1, \; t_{k+1} = \frac{1+\sqrt{1 + 4t_k^2}}{2}.
\]

\begin{fact}
    If $\{x_k\}$ is generated by \ref{eq:FISTA_Alg}, then for any $N \ge 1$, the following bound is satisfied for any minimizer $x_\star$ \cite[Theorem 4.4]{beck2009fast}
    \[
    f(x_{N}) + g(x_{N}) - f(x_\star) - g(x_\star) \le \frac{2L \|x_1 - x_\star\|^2}{N^2}
    \]
\end{fact}

Nesterov acceleration can be applied to the corresponding dual of \eqref{eq:problem}, since the dual involves the sum of a convex and a smooth, convex function. However, the quantity that is reduced at an accelerated rate in this case is the dual-function-value suboptimality, and this is different from the distance to the solution.

\paragraph{Halpern Acceleration.}
The Halpern acceleration \cite{halpern1967fixed}, also referred to as the anchor acceleration in the modern literature, reduces the squared fixed-point residual at an accelerated $\mathcal{O}(1/N^2)$ rate \cite{kim2021accelerated,lieder2021convergence}.
Specifically, suppose $\T\colon \mathbb{R}^d \to \mathbb{R}^d$ is a non-expansive operator. Then, the optimized Halpern method (OHM) 
\[
x_{k+1} = \frac{1}{k+2}x_0 +
\frac{k+1}{k+2}
\T x_{k}
\]
exhibits the rate 
\[
\|\T x_{N-1} - x_{N-1}\|^2\le \frac{4\|x_0-x_\star\|^2}{N^2},
\]
where $x_\star \in \mathrm{Fix} \; \T$ (i.e. $x_\star = \T x_\star$).

OHM can be applied directly to Problem \eqref{eq:problem} by taking $\T$ to be the Peaceman--Rachford splitting (PRS) operator. This algorithm is
\begin{equation}
    \begin{aligned} 
    & y_{k+1} = 2\prox_{\alpha f}(x_k) - x_k\\
    & x_{k+1} = \frac{1}{k+2}x_0 + \frac{k+1}{k+2} ( 2\prox_{\alpha g}(y_{k+1}) - y_{k+1}), \\
    \end{aligned}
    \tag{\textup{OHM}}\label{eq:OHM}
\end{equation}
where $\alpha > 0$. This yields the rate
    \[
    \|\T_{\mathrm{PRS}} x_{N-1} - x_{N-1}\|^2 = \mc{O}(1/N^2),
    \]
    where $\T_{\mathrm{PRS}} = (2 \prox_{\alpha g} - \varmathbb{I}) \; \circ \; (2 \prox_{\alpha f} - \varmathbb{I})$ is the PRS operator.

However, the quantity that is reduced at an accelerated rate in this case is the fixed-point residual of the PRS operator, and this is different from the distance to the solution. 

\paragraph{Chambolle--Pock--Davis--Yin Acceleration.}
% The Chambolle--Pock--Davis--Yin acceleration was first introduced by

The accelerated Chambolle and Pock method \cite[Alg.~2]{Chambolle2010} applied to \eqref{eq:problem} is 
\begin{equation}
    \begin{aligned} 
    & u_{k+1} = u_k - \sigma_kz_k + \sigma_k \prox_{f/\sigma_k}(z_k-(1/\sigma_k)u_k) \\
    & x_{k+1} = \prox_{\tau_k g}(x_k + \tau_k u_{k+1}) \\
    & z_{k+1} = x_{k+1} + \theta_k (x_{k+1}-x_k), \\
    \end{aligned}
    \tag{\textup{CP}}\label{eq:CP}
\end{equation}
where
\[
 \tau_{k+1}=\theta_k \tau_k, \; \sigma_{k+1} = \sigma_k / \theta_k\;, \theta_k = \frac{1}{\sqrt{1+2\mu \tau_k}},
\]
and $\tau_0, \sigma_0 > 0$ with $\tau_0 \sigma_0 \leq 1$, and $z_0 = x_0$.
Similarly, the accelerated Davis--Yin method \cite[Alg.~2]{davis2017three} applied to \eqref{eq:problem} is 
\begin{equation}
\begin{aligned} 
& x_0 = \prox_{\gamma_0 g}(y_0) \\
& u_0 = \frac{1}{\gamma_0}(y_0-x_0) \\
& x_k = \prox_{\gamma_{k-1} g}(y_{k-1} + \gamma_{k-1} u_{k-1}) \\
& u_k = \frac{1}{\gamma_{k-1}}(y_{k-1} + \gamma_{k-1}u_{k-1}-x_k) \\
& y_k = \prox_{\gamma_{k} f}(x_{k} - \gamma_{k} u_{k}), \\
\end{aligned}
\tag{\textup{DYS}}\label{eq:DYS}
\end{equation}
where $y_{0}\in \mathbb{R}^d$ is an initialization, and the proximal stepsizes are given by
\[
 \gamma_{k+1}=\frac{\gamma_k}{\sqrt{1 + 2\gamma_k \mu}}\;\text{for }k=0,\ldots,N-1,
\]
where $\gamma_0 \in (0,\infty)$.

\begin{fact}
If $\{x_k\}$ is generated by \ref{eq:CP} with $\tau_0\sigma_0=1$, then for any $\epsilon>0$, there exists $N_0$ such that for all $N\ge N_0$ \cite[Theorem~2]{Chambolle2010},
    
\begin{align*}
\|x_N-x_\star\|^2
&\le
\frac{1+\epsilon}{N^2}
\left(
\frac{\|x_0-x_\star\|^2}{\mu^2\tau_0^2}
+
\frac{\|u_0-u_\star\|^2}{\mu^2}
\right),
\end{align*}
where $x_\star$ is the minimizer of Problem \eqref{eq:problem} and $u_\star \in \partial g(x_\star)$, equivalently $-u_\star \in \partial f(x_\star)$, is a corresponding dual solution.

\end{fact}

\begin{fact}
If $\{x_k\}$ is generated by \ref{eq:DYS} and $x_\star$ is the minimizer of Problem \eqref{eq:problem}, then $\|x_N-x_\star\|^2 = \mc{O}(1/N^2)$ \cite[Theorem~1.2]{davis2017three}.
\end{fact}

Conceptually, our method in Section~\ref{s:FDR} can be viewed as a Chambolle--Pock--Davis--Yin-type acceleration, but further refined to achieve the optimal constant in the leading order term.

\subsection{Prior works}
Several prior works have sharpened the convergence theory for Douglas--Rachford splitting \cite{douglas1956numerical,lions1979splitting} and its variants under additional regularity assumptions on the functions in \eqref{eq:problem} or operators in \eqref{eq:mi}. Giselsson and Boyd \cite{giselsson2016linear} prove a tight global linear convergence of Douglas--Rachford splitting when one of the functions is both strongly convex and smooth. Davis and Yin \cite{davis2017faster} establish improved convergence rates for Douglas--Rachford splitting and ADMM under several regularity assumptions, including strong convexity, smoothness, and linear regularity. Giselsson \cite{giselsson2017tight} proves a tight global linear convergence rate for Douglas--Rachford splitting applied to Problem \eqref{eq:mi} assuming one operator is strongly monotone and cocoercive and when one operator is strongly monotone and the other is cocoercive. Ryu et al. \cite{ryu2020operator} develop a performance-estimation framework for operator splitting methods and use it to derive tight contraction factors for Douglas--Rachford splitting in monotone inclusion problems like Problem \eqref{eq:mi} under assumptions such as strong monotonicity, cocoercivity, and Lipschitz continuity. Patrinos et al. \cite{patrinos2014douglas} analyze DRS through the Douglas--Rachford envelope and derive accelerated variants in that framework, although their assumptions and performance measures differ from ours. ADMM is an algorithm closely related to DRS, and its linear convergence was shown by {Fran{\c{c}}a} and Bento \cite{francca2016explicit} and Deng and Yin \cite{deng2016global}. More recently, {Brice\~no-Arias} and {Rold{\'a}n} \cite{briceno2026optimal} proposed a modified Peaceman--Rachford splitting method, a closely related variant of Douglas--Rachford splitting \cite{eckstein1992douglas}, for problems in which the two functions have combinations of strong convexity and smoothness. The proposed method achieves an improved linear convergence rate relative to previously known bounds in that setting. 

If one term in \eqref{eq:problem} is additionally smooth, the problem can be treated by forward--backward splitting \cite{bruck1977weak,passty1979ergodic} and accelerated variants such as FISTA \cite{beck2009fast}. {Brice\~no-Arias} \cite{briceno2025lyapunov} studies FISTA for this smooth composite setting when both smooth and nonsmooth terms may contribute strong convexity and shows that the sharpest theoretical rate is obtained by leveraging all available strong convexity through the smooth term.

These works differ from the present paper in that they rely on smoothness or other regularity assumptions to obtain linear convergence rates. In contrast, we consider the nonsmooth proximable setting where one function is strongly convex, establish an accelerated $O(1/N^2)$ rate for the squared distance to the solution, and prove that this rate and its leading constant are optimal.

\section{Fast Douglas--Rachford Splitting}
\label{s:FDR}
We now present Fast Douglas--Rachford Splitting (FDR):
\begin{equation}
\begin{aligned} 
& w_0 = x_0 - \eta_0 u_0 \\
& y_{k+1}=\prox_{\eta_k g}\left(2x_k - w_k\right)\\
& w_{k+1}=\left(1 + \frac{\eta_{k+1}}{\eta_k}\right)y_{k+1} - \left(\frac{\eta_{k+1}}
{\eta_k}\right)(2x_{k} - w_k) \\
& x_{k+1}=\prox_{\eta_{k+1} f}\left(w_{k+1}\right)\\
\end{aligned}
\tag{\textup{FDR}}\label{eq:FDR_Alg}
\end{equation}
for $k=0,\dots,N-1$, where $N\in \mathbb{N}$ is a pre-specified total iteration count, $x_{0}, u_0\in \mathbb{R}^d$ are respectively primal and dual starting points, and the proximal stepsizes are given by
\[
 \eta_{k}=\frac{2N\mu}{1 + 4kN\mu^2}\qquad\text{for }k=0,\ldots,N.
\]

In Section~\ref{ss:conv-proof}, we establish the following convergence guarantee for FDR.
\begin{theorem}
\label{thm:FDR-rate}
Let $N\ge1$ and let $f\colon \mathbb{R}^{d}\to\mathbb{R}\cup\{\infty\}$ be closed, convex, and proper, and $g\colon\mathbb{R}^{d}\to\mathbb{R}\cup\{\infty\}$ be closed, proper, and $\mu$-strongly convex with $\mu > 0$.  Assume that a primal solution $x_\star\in \argmin(f+g)$ and a dual solution $u_\star$ satisfying $u_\star \in \partial g(x_\star)$, $-u_\star \in \partial f(x_\star)$ exist.
Then, the iterates generated by \ref{eq:FDR_Alg} satisfy the rate
\[
\|x_N-x_\star\|^2\leq\frac{\|x_0-x_\star\|^2 + \|u_0 - u_\star\|^2}{1+4N^2\mu^2}.
\]
\end{theorem}

\paragraph{Discussion.}
For FDR, the total number of iterations $N$ must be specified \textit{a priori}, as the proximal step-size sequence $\{\eta_k\}_{k=0}^N$ explicitly depends on $N$. Additionally, Theorem~\ref{thm:FDR-rate} provides a guarantee on the final iterate, $x_N$. Both properties are shared by exact optimal methods such as OGM and OptISTA. 

We observe that if the step-sizes are constant ($\eta_0 = \eta_1 = \cdots = \eta_N$), our algorithm reduces to Peaceman--Rachford splitting up to our choice of initialization \cite{peaceman1955numerical}.

Moreover, as $N, k \to \infty$ the ratio $\eta_{k+1}/\eta_{k} \to 1$, so the later FDR iterations resemble Peaceman--Rachford splitting with varying, decreasing proximal stepsizes.

Theorem~\ref{thm:FDR-rate} shows that FDR improves the leading asymptotic constant of the accelerated Chambolle--Pock and Davis--Yin methods by a factor of four. Details are provided in Appendix \ref{apx:leading_constants}.

\subsection{Proof of Theorem~\ref{thm:FDR-rate}}
\label{ss:conv-proof}
We establish the convergence rate of FDR with a Lyapunov analysis, which is a widely used proof template in the analysis of first-order methods \cite{jang2025computer,kim2016optimized,park2023factor,park2022exact}.
For notational simplicity, define
\[
R^2=\|x_0-x_\star\|^2 + \|u_0 - u_\star\|^2.
\]

% To analyze FDR we choose the following Lyapunov sequence $\mc{V}_k$ with indices $k = -1, 0, \ldots, N$.

Specifically, define
\begin{align*}
\mathcal{V}_{-1}&=\nu R^2 - \nu^2\|-2N\mu(x_0-x_\star) + (u_0-u_\star)\|^2\\
\mathcal{V}_0&=\nu^2\|x_\star+2N\mu u_\star-x_0-2N\mu u_0\|^2
\\
\mathcal{V}_k&=\nu^2\|(1 + 4kN\mu^2)x_\star+2N\mu u_\star-(1 + 4kN\mu^2)(2x_k-w_k)\|^2 \quad \forall k = 1, \ldots, N-1
\\
\mathcal{V}_N&=\|x_N-x_\star\|^2+4N^2\mu^2\nu^2\|u_\star + f'(x_N)\|^2,
\end{align*}
where $x_k$ and $w_k$ are generated by \eqref{eq:FDR_Alg}, $f'(x_N) = (w_N-x_N) / \eta_N$, and
\[\nu = \frac{1}{1 + 4N^2 \mu^2}.\]
Once we establish that the Lyapunov sequence is dissipative, i.e., that
\[
\mc{V}_N \le \mc{V}_{N-1} \le \cdots \le \mc{V}_1 \le \mc{V}_0 \le \mc{V}_{-1},
\]
it follows immediately that
\[
\|x_N - x_\star\|^2 \le \mc{V}_N \le \cdots \le \mc{V}_{-1} \le\nu R^2 = \frac{\|x_0-x_\star\|^2 + \|u_0 - u_\star\|^2}{1 + 4N^2\mu^2},
\]
which implies the stated rate
\[
\|x_N - x_\star\|^2 \le \frac{\|x_0-x_\star\|^2 + \|u_0 - u_\star\|^2}{1 + 4N^2\mu^2}.
\]

It remains to show that the Lyapunov sequence is non-increasing.
% Note that for $N=1$, the inequality $\mc{V}_1 \le \mc{V}_0$ is shown by Case 4 and not Case 2.

\underline{Case 1}:\;$\mc{V}_0\le\mc{V}_{-1}$
\begin{align*}
    \mc{V}_0-\mc{V}_{-1} &= \nu^2\|x_\star+2N\mu u_\star-x_0-2N\mu u_0\|^2 - \nu R^2 + \nu^2\|-2N\mu(x_0-x_\star) + (u_0-u_\star)\|^2 \\
    &= \nu^2\|x_\star-x_0+2N\mu (u_\star-u_0)\|^2 + \nu^2\|2N\mu(x_\star-x_0) + (u_0-u_\star)\|^2 - \nu R^2 \\
    &= \nu^2\left[(1+4N^2\mu^2)\|x_0-x_\star\|^2+(1+4N^2\mu^2)\|u_0-u_\star\|^2\right] -\nu R^2 \\
    &= \nu (\|x_0-x_\star\|^2+\|u_0-u_\star\|^2-R^2) \\
    &= 0 
\end{align*}

\underline{Case 2}:\;$\mc{V}_1\le\mc{V}_{0}$ for $N \ge 2$

First, we will give some formulas and definitions.
\begin{subequations}
\begin{align}
    \|a\|^2-\|b\|^2 &= \langle a+b,a-b\rangle \label{eq:polarization-identity}  \\
    \Omega &= 1+4N\mu^2 \\
    \Omega \eta_{1} &= \eta_0 \label{eq:eta1-2} \\
    g'(y_1) &= \frac{ x_0 + \eta_0 u_0-y_1}{\eta_0} \label{eq:y1} \\
    f'(x_{1}) &= \frac{ w_1 - x_1}{\eta_1} \label{eq:x1} \\
    f'(x_{1})+g'(y_1) &= \frac{(y_1-x_1)\Omega}{\eta_0} \label{eq:fp-gp-x1y1} \\
    2\eta_0(1+\Omega)\mu &= \Omega^2-1
\end{align}
\end{subequations}

\begingroup
\allowdisplaybreaks
\begin{align*}
    \mc{V}_1-\mc{V}_{0} &= \nu^2[\|\Omega x_\star + 2N\mu u_\star-\Omega(2x_1-w_1)\|^2 - \|x_\star + 2N\mu u_\star - x_0 - 2N\mu u_0\|^2] \\
    &= \nu^2[\|\Omega x_\star+\eta_0 u_\star-\Omega(x_1-\eta_1 f'(x_1))\|^2 - \\ &\|x_\star + \eta_0 u_\star - (y_1 + \eta_0 g'(y_1))\|^2] {\color{gray}\qquad\rhd\,\textrm{using } \eqref{eq:y1} \textrm{ and } \eqref{eq:x1}} \\
    &= \nu^2[\langle\Omega x_\star + \eta_0u_\star-\Omega x_1 + \Omega \eta_1 f'(x_1)+x_\star+\eta_0 u_\star - y_1 - \eta_0g'(y_1) ,\\
    & \Omega x_\star + \eta_0u_\star-\Omega x_1 + \Omega \eta_1 f'(x_1)-x_\star-\eta_0 u_\star + y_1 + \eta_0g'(y_1) \rangle] {\color{gray}\qquad\rhd\,\textrm{using } \eqref{eq:polarization-identity}} \\
    &= \nu^2[\langle (1+\Omega)x_\star - y_1 - \Omega x_1 + 2\eta_0 u_\star + \Omega \eta_1 f'(x_1) - \eta_0 g'(y_1), \\
    & (-1+\Omega)x_\star + y_1 - \Omega x_1 + \Omega \eta_1 f'(x_1) + \eta_0 g'(y_1) \rangle ] \\
    &= \nu^2[\langle (1+\Omega)x_\star - y_1 - \Omega x_1 + 2\eta_0 u_\star + \eta_0(f'(x_1)-g'(y_1)), \\
    & (-1+\Omega)x_\star + y_1 - \Omega x_1 + (y_1-x_1)\Omega \rangle ] {\color{gray}\qquad\rhd\,\textrm{using } \eqref{eq:eta1-2} \textrm{ and } \eqref{eq:fp-gp-x1y1}} \\
    &= \nu^2[\langle (1+\Omega)x_\star - y_1 - \Omega x_1 + 2\eta_0 u_\star + \eta_0(f'(x_1)-g'(y_1)), (1+\Omega)(y_1-x_\star) + 2 \Omega (x_\star-x_1) \rangle ] \\
    &= \nu^2[ \langle (1+\Omega)(y_1-x_\star), (1+\Omega)x_\star - y_1 - \Omega x_1 + 2\eta_0 u_\star + \eta_0(f'(x_1)-g'(y_1)) \rangle + \\
    & \langle 2 \Omega (x_\star-x_1), (1+\Omega)x_\star - y_1 - \Omega x_1 + 2\eta_0 u_\star + \eta_0(f'(x_1)-g'(y_1)) \rangle] \\
    &= \nu^2[ \langle (1+\Omega)(y_1-x_\star), (1+\Omega)x_\star - y_1 - \Omega x_1 + 2\eta_0 u_\star + (y_1-x_1)\Omega - 2\eta_0 g'(y_1) \rangle + \\
    & \langle 2 \Omega (x_\star-x_1), (1+\Omega)x_\star - y_1 - \Omega x_1 + 2\eta_0 u_\star + 2\eta_0 f'(x_1) - (y_1-x_1)\Omega \rangle] {\color{gray}\qquad\rhd\,\textrm{using } \eqref{eq:fp-gp-x1y1}}\\
    &= \nu^2[ \langle y_1-x_\star, (1+\Omega)^2x_\star + (\Omega^2-1)y_1 - 2\Omega(1+\Omega) x_1 + 2\eta_0 (1+\Omega) (u_\star- g'(y_1)) \rangle + \\
    & \langle x_\star-x_1, 2\Omega(1+\Omega)x_\star -2\Omega (1+\Omega)y_1 + 4\eta_0 \Omega(u_\star + f'(x_1)) \rangle] \\
    &= \nu^2[ \langle y_1-x_\star, (1+\Omega)^2x_\star + (\Omega^2-1)y_1 - 2\Omega(1+\Omega) x_\star + 2\eta_0 (1+\Omega) (u_\star- g'(y_1)) \rangle + \\
    & \langle x_\star-x_1, 4\eta_0 \Omega(u_\star + f'(x_1)) \rangle] \\
    &= \nu^2[ \langle y_1-x_\star, (\Omega^2-1)(y_1 - x_\star) + 2\eta_0 (1+\Omega) (u_\star- g'(y_1)) \rangle -4\eta_0 \Omega \langle x_\star-x_1, -u_\star - f'(x_1) \rangle] \\
    &= \nu^2\{ 2\eta_0 (1+\Omega)[-\langle y_1-x_\star, g'(y_1) - u_\star \rangle + \mu \|y_1-x_\star\|^2] -4\eta_0 \Omega \langle x_\star-x_1, -u_\star - f'(x_1) \rangle\} \\
    &\leq 0,
\end{align*}
\endgroup
where the final inequality holds by monotonicity of $\partial f$, $\mu$-strong monotonicity of $\partial g$, since $-u_\star \in \partial f(x_\star)$, and because $2\eta_0 (1+\Omega)\ge 0$ and $4\eta_0 \Omega\ge 0$.

\underline{Case 3}:\;$\mc{V}_{k+1}\le\mc{V}_{k} \quad \forall k = 1, \ldots, N-2$

First we will give some important formulas and definitions.

\begin{subequations}
\begin{align}
    g'(y_{k+1}) &= \frac{2x_{k}-w_k-y_{k+1}}{\eta_k} \label{eq:ykp1} \\
    f'(x_{k+1}) &= \frac{w_{k+1}-x_{k+1}}{\eta_{k+1}} \label{eq:xkp1} \\
    \Delta_{k} &= 1 + 4kN\mu^2 \\
    \Delta_{k+1} &= 1 + 4(k+1)N\mu^2 \\
    \eta_{k}\Delta_{k} &= \eta_{k+1}\Delta_{k+1} = 2N\mu \label{eq:delta-eta-identity} \\
    f'(x_{k+1}) + g'(y_{k+1}) &= \frac{(y_{k+1}-x_{k+1})\Delta_{k+1}}{2N \mu} \label{eq:fp-gp-identity}
\end{align}
\end{subequations}

\begingroup
\allowdisplaybreaks
\begin{align*}
     \mc{V}_{k+1}-\mc{V}_{k} &= \nu^2 [\|\Delta_{k+1}x_\star + 2N\mu u_\star - \Delta_{k+1}(2x_{k+1}-w_{k+1})\|^2 \\ 
     &- \|\Delta_{k}x_\star + 2N\mu u_\star - \Delta_{k}(2x_{k}-w_k)\|^2] \\
    &= \nu^2 [\|\Delta_{k+1}x_\star + 2N\mu u_\star - \Delta_{k+1}(x_{k+1}-\eta_{k+1}f'(x_{k+1}))\|^2 \\ 
     &- \|\Delta_{k}x_\star + 2N\mu u_\star - \Delta_{k}(y_{k+1} + \eta_{k}g'(y_{k+1}))\|^2] {\color{gray}\qquad\rhd\,\textrm{from } \eqref{eq:ykp1} \textrm{ and } \eqref{eq:xkp1}} \\
    &= \nu^2 [\langle (\Delta_{k+1}+\Delta_k)x_\star+4N \mu u_\star - \Delta_{k+1}(x_{k+1} - \eta_{k+1}f'(x_{k+1})) - \Delta_k(y_{k+1} + \eta_k g'(y_{k+1})), \\
    &(\Delta_{k+1}-\Delta_k)x_\star - \Delta_{k+1}(x_{k+1}-\eta_{k+1}f'(x_{k+1}))+\Delta_k(y_{k+1}+\eta_kg'(y_{k+1}))\rangle] {\color{gray}\qquad\rhd\,\textrm{from } \eqref{eq:polarization-identity}} \\
    &= \nu^2 [\langle \Delta_{k+1}(x_\star-x_{k+1})+\Delta_k(x_\star -y_{k+1}) + 4N\mu u_\star+\Delta_{k+1} \eta_{k+1}f'(x_{k+1})-\Delta_k \eta_k g'(y_{k+1}),\\
    & \Delta_{k+1}(x_\star-x_{k+1})-\Delta_k(x_\star -y_{k+1}) +\Delta_{k+1} \eta_{k+1}f'(x_{k+1})+\Delta_k \eta_k g'(y_{k+1}) \rangle] \\
    &= \nu^2[\Delta_{k+1}^2\|x_\star-x_{k+1}\|^2 - \Delta_k^2\|x_\star - y_{k+1}\|^2 + 2\Delta_{k+1}^2 \langle x_\star - x_{k+1}, \eta_{k+1}f'(x_{k+1})\rangle \\
    &+2\Delta_k^2 \langle x_\star - y_{k+1}, \eta_k g'(y_{k+1})\rangle + \Delta_{k+1}^2\eta_{k+1}^2 \|f'(x_{k+1})\|^2 - \Delta_k^2 \eta_k ^2 \|g'(y_{k+1})\|^2 \\
    &+ \langle 4N\mu u_\star,\Delta_{k+1}(x_\star-x_{k+1})-\Delta_k(x_\star -y_{k+1}) +\Delta_{k+1} \eta_{k+1}f'(x_{k+1})+\Delta_k \eta_k g'(y_{k+1})\rangle] \\
    &= \nu^2 [\langle x_\star - x_{k+1}, 2\Delta_{k+1}^2\eta_{k+1}f'(x_{k+1})+4\Delta_{k+1}N\mu u_\star + \Delta_{k+1}^2(x_\star - x_{k+1})\rangle \\
    &+ \langle x_\star - y_{k+1}, 2\Delta_k^2 \eta_k g'(y_{k+1})-4\Delta_k N\mu u_\star - \Delta_k^2 (x_\star - y_{k+1})\rangle \\
    &+4N^2\mu^2\langle 2u_\star + f'(x_{k+1}) - g'(y_{k+1}), f'(x_{k+1}) + g'(y_{k+1}) \rangle ] {\color{gray}\qquad\rhd\,\textrm{from } \eqref{eq:delta-eta-identity}} \\
    &= \nu^2 [\langle x_\star - x_{k+1}, 2\Delta_{k+1}^2\eta_{k+1}f'(x_{k+1})+4\Delta_{k+1}N\mu u_\star + \Delta_{k+1}^2(x_\star - x_{k+1})\rangle \\
    &+ \langle x_\star - y_{k+1}, 2\Delta_k^2 \eta_k g'(y_{k+1})-4\Delta_k N\mu u_\star - \Delta_k^2 (x_\star - y_{k+1})\rangle \\
    &+2N\mu \Delta_{k+1}\langle 2u_\star + f'(x_{k+1}) - g'(y_{k+1}), y_{k+1} - x_{k+1} \rangle ] {\color{gray}\qquad\rhd\,\textrm{from } \eqref{eq:fp-gp-identity}} \\
    &= \nu^2 [\langle x_\star - x_{k+1}, 6N\Delta_{k+1}\mu f'(x_{k+1})+8N\Delta_{k+1}\mu u_\star - 2N\Delta_{k+1}\mu g'(y_{k+1}) + \\
    &\Delta_{k+1}^2(x_\star - x_{k+1})\rangle + \langle x_\star - y_{k+1}, 2N\mu(2\Delta_k + \Delta_{k+1}) g'(y_{k+1})-4N\mu(\Delta_k + \Delta_{k+1}) u_\star \\
    &- 2N\mu \Delta_{k+1}f'(x_{k+1}) - \Delta_k^2 (x_\star - y_{k+1})\rangle] {\color{gray}\qquad\rhd\,\textrm{from } \eqref{eq:delta-eta-identity}} \\
    &= \nu^2 [\langle x_\star - x_{k+1}, 8N\Delta_{k+1}\mu (f'(x_{k+1})+ u_\star)+\Delta_{k+1}^2(x_\star - y_{k+1})\rangle \\
    &+ \langle x_\star - y_{k+1}, 4N\mu(\Delta_k + \Delta_{k+1}) g'(y_{k+1})-4N\mu(\Delta_k + \Delta_{k+1}) u_\star + (x_{k+1}-y_{k+1})\Delta_{k+1}^2 \\
    &  - \Delta_k^2 (x_\star - y_{k+1})\rangle] {\color{gray}\qquad\rhd\,\textrm{from } \eqref{eq:fp-gp-identity}} \\
    &= \nu^2 [-8N\Delta_{k+1}\mu\langle x_{k+1} - x_\star, f'(x_{k+1}) + u_\star\rangle \\
    &+ \langle x_\star - y_{k+1}, 4N\mu(\Delta_k + \Delta_{k+1}) (g'(y_{k+1}) - u_\star) +  (\Delta_{k+1}^2-\Delta_k^2) (x_\star - y_{k+1})\rangle] \\
    &= \nu^2 \{-8N\Delta_{k+1}\mu\langle x_{k+1} - x_\star, f'(x_{k+1}) + u_\star\rangle \\
    &+ 4N\mu(\Delta_k + \Delta_{k+1}) [-\langle x_\star - y_{k+1},  u_\star - g'(y_{k+1}) \rangle + \mu \|x_\star - y_{k+1}\|^2]\} \\
    & \le 0
\end{align*}
\endgroup
where the final inequality holds by monotonicity of $\partial f$, $\mu$-strong monotonicity of $\partial g$, since $-u_\star \in \partial f(x_\star)$, and because the coefficients $8N\Delta_{k+1} \mu$ and $4N \mu (\Delta_k + \Delta_{k+1})$ are non-negative.

\underline{Case 4}:\;$\mc{V}_{N}\le\mc{V}_{N-1}$
(including the case $\mc{V}_{1}\le\mc{V}_{0}$ when $N=1$)

First, we will give some formulas and definitions.

\begin{subequations}
\begin{align}
    g'(y_{N}) &= \frac{2x_{N-1}-w_{N-1}-y_{N}}{\eta_{N-1}} \label{eq:yN} \\
    f'(x_{N}) &= \frac{w_{N}-x_{N}}{\eta_{N}} \label{eq:xN} \\
    \Delta_{N-1} &= 1+4(N-1)N\mu^2 \\
    \Delta_{N} &= 1+4N^2\mu^2 \\
    \Delta_{N-1}\eta_{N-1} &= \Delta_{N}\eta_{N} = 2N\mu \label{eq:delta-eta} \\
    f'(x_{N}) + g'(y_{N}) &= \frac{(y_{N}-x_{N})\Delta_{N}}{2N \mu} \label{eq:fpN-gpN-identity}
\end{align}
\end{subequations}

\begingroup
\allowdisplaybreaks
\begin{align*}
     \mc{V}_N-\mc{V}_{N-1} &= \|x_N-x_\star\|^2+4N^2\mu^2\nu^2\|u_\star + f'(x_N)\|^2  - \nu^2\|\Delta_{N-1} x_\star+2N\mu u_\star-\Delta_{N-1}(2x_{N-1}-w_{N-1})\|^2 \\
&= \|x_N-x_\star\|^2+4N^2\mu^2\nu^2\|u_\star + f'(x_N)\|^2  - \nu^2\|\Delta_{N-1} x_\star+2N\mu u_\star-\Delta_{N-1}(y_N+\eta_{N-1}g'(y_N))\|^2 \\
&= \|x_N-x_\star\|^2+4N^2\mu^2\nu^2\|u_\star + f'(x_N)\|^2  - \nu^2\|\Delta_{N-1} (x_\star-y_N)+2N\mu (u_\star-g'(y_N))\|^2 {\color{gray}\qquad\rhd\,\textrm{from } \eqref{eq:delta-eta}} \\
&= \|x_N-x_\star\|^2+\nu^2 \langle 2N\mu(u_\star+f'(x_N)) + \Delta_{N-1}(x_\star - y_N) + 2N\mu(u_\star - g'(y_N)), \\
& 2N\mu(u_\star+f'(x_N)) - \Delta_{N-1}(x_\star - y_N) - 2N\mu(u_\star - g'(y_N)) \rangle {\color{gray}\qquad\rhd\,\textrm{from } \eqref{eq:polarization-identity}} \\
&= \|x_N-x_\star\|^2+\nu^2 \langle 2N\mu(u_\star+f'(x_N)) + \Delta_{N-1}(x_\star - y_N) + 2N\mu(u_\star - g'(y_N)), \\
& (y_N-x_N)\Delta_N-\Delta_{N-1}(x_\star-y_N) \rangle {\color{gray}\qquad\rhd\,\textrm{from } \eqref{eq:fpN-gpN-identity}} \\
&= \|x_N-x_\star\|^2+\nu^2 \langle 2N\mu(u_\star+f'(x_N)) + \Delta_{N-1}(x_\star - y_N) + 2N\mu(u_\star - g'(y_N)), \\
& (\Delta_{N-1}+\Delta_N)(y_N-x_\star) - \Delta_N(x_N-x_\star) \rangle {\color{gray}\qquad\rhd\,\textrm{add and subtract } \Delta_N x_\star} \\
&= \nu^2 \langle 2N\mu(u_\star+f'(x_N)) - \Delta_{N-1}(y_N - x_\star) + 2N\mu(u_\star - g'(y_N)), (\Delta_{N-1}+\Delta_N)(y_N-x_\star) \rangle + \\
& \nu^2 \langle -\Delta_N(x_N-x_\star) + 2N\mu(u_\star+f'(x_N)) - \Delta_{N-1}(y_N - x_\star) + 2N\mu(u_\star - g'(y_N)), -\Delta_N(x_N-x_\star) \rangle \\
&= \nu^2 \langle 2N\mu u_\star + (y_N-x_N)\Delta_N - 2N\mu g'(y_N) - \Delta_{N-1}(y_N - x_\star) + 2N\mu(u_\star - g'(y_N)), \\
& (\Delta_{N-1}+\Delta_N)(y_N-x_\star) \rangle + \nu^2 \langle -\Delta_N(x_N-x_\star) + 2N\mu(u_\star+f'(x_N)) - \Delta_{N-1}(y_N - x_\star) + \\
&2N\mu u_\star + 2N \mu f'(x_N) + (x_N-y_N)\Delta_N, -\Delta_N(x_N-x_\star) \rangle {\color{gray}\qquad\rhd\,\textrm{from } \eqref{eq:fpN-gpN-identity}} \\
&= \nu^2 \langle 4N\mu (u_\star - g'(y_N)) + (\Delta_N - \Delta_{N-1})(y_N-x_\star) + (x_\star - x_N)\Delta_N, (\Delta_{N-1}+\Delta_N)(y_N-x_\star) \rangle + \\
&\nu^2 \langle 4N\mu (u_\star + f'(x_N)) - (\Delta_N + \Delta_{N-1})(y_N-x_\star), -\Delta_N(x_N-x_\star) \rangle \\
&= \nu^2 \langle 4N\mu (u_\star - g'(y_N)) + (\Delta_N - \Delta_{N-1})(y_N-x_\star), (\Delta_{N-1}+\Delta_N)(y_N-x_\star) \rangle + \\
&\nu^2 \langle 4N\mu (u_\star + f'(x_N)), -\Delta_N(x_N-x_\star) \rangle \\
&= 8N\mu(1+2N(2N-1)\mu^2 )\nu^2(-\langle g'(y_N)-u_\star, y_N-x_\star \rangle + \mu \|y_N-x_\star\|^2) + \\
& 4N\mu \nu(-\langle f'(x_N)+u_\star, x_N-x_\star \rangle) \\
& \le 0
\end{align*}
\endgroup
where the final inequality holds by monotonicity of $\partial f$, $\mu$-strong monotonicity of $\partial g$, since $-u_\star \in \partial f(x_\star)$, and because $8N\mu(1+2N(2N-1)\mu^2 )\nu^2\ge 0$ and $4N\mu \nu\ge 0$.

\qed

\section{Matching lower bound}\label{sec:lower-bound}
In this section, we present a complexity lower bound that matches the $\mc{O}(1/N^2)$ rate and its leading constant of Theorem~\ref{thm:FDR-rate}. 

% This is achieved by constructing worst-case functions $g$ and $f$ for which no deterministic $N$-step prox--prox method can attain a faster decay of the distance to the solution.

We follow a standard approach: we first construct a pair of  worst-case functions for the class of methods satisfying the \emph{proximal span condition}, and then apply a resisting oracle technique to extend the lower bound to apply to all deterministic algorithms. 

\subsection{Proximal span condition}
Intuitively, for first-order methods, the span condition is defined by requiring that each iterate belongs to the span of the previous iterates and the associated first-order information. This notion has been formalized in various forms in prior works \cite{nemirovsky1991optimality,carmon2020lower,drori2017exact,ouyang2021lower,drori2022oracle,park2022exact,jang2025computer}.

In our particular setting, we consider composite problems accessed through proximal oracles. Before giving the formal definition in Section~\ref{sec:resisting}, we use the term \emph{deterministic $N$-step
prox-prox method} informally to refer to any deterministic algorithm that makes a total of $2N$ proximal oracle calls, each call being to either $f$ or $g$, and then outputs a point $x_N$. Since the method has access to two starting points, namely $x_0$ and $u_0$, the corresponding \emph{proximal span condition} can be stated as below for fixed iteration count $N>0$. Specifically, for fixed $(x_0,u_0)$, we say a \emph{trajectory} or a \emph{sequence}
\[
\big\{(z_i,\delta_i,\gamma_i)\big\}_{i=0}^{2N-1},\  x_{N} \big\}
\]
satisfies the \emph{proximal span condition} if
\begin{alignat*}{3}
&\delta_i\in \{0,1\} &&\textup{ for }i=0,\dots,2N-1,\\
&z_{0}\in x_0+\mathrm{span}\{u_0\}\\
&d_i=\left\{
\begin{array}{ll}
z_i-\prox_{\gamma_i f}(z_i) &\textup{for some $\gamma_i>0$, if $\delta_i=0$, }\\
 z_i-\prox_{\gamma_i g}(z_i) &\textup{for some $\gamma_i>0$, if $\delta_i=1$,}
\end{array}
\right.
&&\textup{ for }i=0,\dots,2N-1,\\
&z_{i}\in x_0+\mathrm{span}\{u_0, d_0,\dots,d_{i-1}\}
&&\textup{ for }i=1,\dots,2N-1,\\
&x_{N}\in x_0+\mathrm{span}\{u_0,d_0,\dots,d_{2N-1}\}
&&
\!\!\!\!\!\!\!\!\!\!\!\!\!\!\!\!\!\!\!\!\!\!\!\!\!\!\!\!\!\!\!\!\!\!\!\!\!\!\!
\end{alignat*}
and we say a \emph{method} satisfies the \emph{proximal span condition} if the trajectory $\big\{(z_i,\delta_i,\gamma_i)\big\}_{i=0}^{2N-1},\  x_{N} \big\}$ produced by the first-order deterministic method satisfies the proximal span condition for any starting point $(x_0,u_0)$. 

To clarify, we make no assumption on either the order of the proximal oracle evaluations or which function is queried at each step. We only assume that the total number of proximal oracle evaluations is $2N$.

Under the proximal span condition, we establish the following lower bound.

\begin{theorem}\label{thm:lbspan}
Fix a positive integer $N$ and let $\mu>0$. For any initial pair
$(x_0,u_0)\in \mathbb{R}^{2N+2}\times \mathbb{R}^{2N+2}$, there exist
a closed, proper, and convex function
$f\colon \mathbb{R}^{2N+2}\to \mathbb{R}\cup\{\infty\}$,
a closed, proper, and $\mu$-strongly convex function
$g\colon \mathbb{R}^{2N+2}\to \mathbb{R}\cup\{\infty\}$,
$x_\star\in\argmin(f+g)$, and
$u_\star\in\partial g(x_\star)$ with $-u_\star\in\partial f(x_\star)$
such that the following holds:

For every deterministic $N$-step prox-prox method satisfying the proximal span condition,
if $x_N$ denotes its final output when run from $(x_0,u_0)$ on the pair $(f,g)$, then
\[
\|x_N-x_\star\|^2
\ge
\frac{\|x_0-x_\star\|^2+\|u_0-u_\star\|^2}
{1+4N^2\mu^2+4N\mu}.
\]
\end{theorem}

\subsection{Proof of Theorem~\ref{thm:lbspan}
% Worst-case instance under the span condition
}\label{sec:lbspan}

We now construct a worst-case instance for first-order deterministic prox-prox methods satisfying the proximal span condition. Throughout this subsection, the starting pair $(x_0,u_0)$ is given and fixed.

Fix $N\in \mathbb{N}$ and $\mu>0$. 
Choose an orthonormal set of vectors
\[
e_{-1},e_0,e_1,\dots,e_{2N}
\]
of $\mathbb{R}^{2N+2}$  such that $u_0 \in \mathrm{span}\{e_{-1}\}$. Let $t\in\mathbb{R}^d$ such that
\[
t :=  \sum_{i=0}^{2N} t_i e_i \in \mathbb{R}^{d}.
\]
where $t_i >0$ for $0\le i \le 2N$.
We define two closed convex sets by
\[
C_t
:=
\Big\{
y\in\mathbb{R}^{2N+2}:
y_{-1}=0,\ y_0=t_0,\ 
(y_{2k-1},y_{2k})\in[(0,0),(t_{2k-1},t_{2k})],\ k=1,\dots,N
\Big\},
\]
and
\[
D_t
:=
\Big\{
y\in\mathbb{R}^{2N+2}:
y_{-1}=0,\ 
(y_{2k},y_{2k+1})\in[(0,0),(t_{2k},t_{2k+1})],\ k=0,\dots,N-1
\Big\},
\]
where $y_i$ is shorthand for $\langle y, e_i\rangle$ for $-1\le i\le 2N$ and $[(0,0),(a,b)]$ is the line segment with endpoints $(0,0)$ and $(a,b)$.
We then define closed, proper, and $\mu$-strongly convex $g$
\[
g(x)
:=
\frac{\mu}{2}\|x-x_0\|^2+\langle u_0,x-x_0\rangle+\delta_{x_0+C_t}(x),
\]
and closed, proper, and convex (possibly non-smooth) $f$ as
\[
f(x)
:=
-\langle u_0,x-x_0\rangle+\delta_{x_0+D_t}(x).
\]
Then minimizing $f+g$ is equivalent to minimizing
\[
\frac{\mu}{2}\|x-x_0\|^2
\]
over $(x_0+C_t)\cap(x_0+D_t)$. 
Because
\[
C_t\cap D_t=\{y\in \mathbb{R}^{2N+2}: y_{-1}=0, \, y_i=t_i \textup{ for } 0\le i \le 2N\},
\]
we have
\[
(x_0+C_t)\cap(x_0+D_t)=\{x_0+y \in\mathbb{R}^{2N+2} :  y_{-1}=0, \, y_i =t_i \textup{ for } 0\le i \le 2N\}.
\]
Hence $x_\star:=x_0 + t$ is the unique minimizer of $f+g$.

Moreover, the proximal mappings admit the explicit representations
\[
\prox_{\gamma f}(x_0+y)
=
x_0+\Pi_{D_t}(y+\gamma u_0),
\]
and
\[
\prox_{\gamma g}(x_0+y)
=
x_0+\Pi_{C_t}\!\left(\frac{y-\gamma u_0}{1+\gamma\mu}\right),
\]
where $\Pi_{D_t}$ and $\Pi_{C_t}$ denote the Euclidean projections onto $D_t$ and $C_t$, respectively.

For $k=-1,0,\dots,2N$, define
\[
S_k:=\mathrm{span}\{e_{-1},e_0,\dots,e_k\},
\qquad
S_{-1}:=\mathrm{span}\{e_{-1}\}.
\]
The sets $C_t$ and $D_t$ are arranged so that the two proximal operators reveal coordinates one at a time. We will refer to this as the \emph{proximal zero-chain property}, and the next lemma formalizes this construction.
\begin{lemma}
For every $\gamma>0$ and every $k=-1,0,\dots,2N-1$,
\[
y\in S_k
\quad\Longrightarrow\quad
\prox_{\gamma f}(x_0+y) = x_0+ \Pi_{D_t}(y+\gamma u_0)\in x_0 + S_{k+1},
\]
and
\[
y\in S_k
\quad\Longrightarrow\quad
\prox_{\gamma g}(x_0+y)= x_0+ \Pi_{C_t}\!\left(\frac{y-\gamma u_0}{1+\gamma\mu}\right)\in x_0 + S_{k+1}.
\]
\end{lemma}

\begin{proof}
Both $C_t$ and $D_t$ are Cartesian products of one-dimensional and two-dimensional blocks, and hence their Euclidean projections act blockwise. 
Since $u_0\in \mathrm{span}\{e_{-1}\}$ and both $C_t$ and $D_t$ impose the constraint $z_{-1}=0$, the shift by $\pm \gamma u_0$ only affects the $e_{-1}$-component and does not reveal any new coordinate. Therefore, if $y\in S_k$, then after projecting onto either $C_t$ or $D_t$, at most the next coordinate $e_{k+1}$ can become nonzero. This proves the claim. \qed
\end{proof}

This lemma shows that if $y\in S_k$, then 
\[
(x_0 + y)-\prox_{\gamma g}(x_0+y) \in S_{k+1}
\]
because each proximal evaluation reveals at most one new coordinate. Consequently, for any prox-prox method satisfying the proximal span condition, the final output $x_{N}$ generated after at most $2N$ proximal evaluations must belong to
\[
x_0+S_{2N-1}.
\]
Since
\[
x_\star-x_0=t=\sum_{i=0}^{2N} t_i e_i
\]
has a nonzero $e_{2N}$-component, we obtain
\[
\inf_{y\in x_0+S_{2N-1}} \|y-x_\star\|^2 = t_{2N}^2>0.
\]
Thus, after at most $2N$ proximal oracle calls, no method satisfying the proximal span condition can guarantee an error smaller than $t_{2N}^2$ on this instance. The next lemma rewrites the optimality conditions at the endpoint $x_\star=x_0+t$ in blockwise form. 

\begin{lemma}\label{lem:normal-cone-conditions}
Let
\[
x_\star=x_0+t = x_0 + \sum_{i=0}^{2N} t_i e_i,
\]
and let
\[
g_\star=\sum_{i=0}^{2N} g_i e_i \in \mathbb{R}^{d}
\]
where $g_i\in \mathbb{R}$ for $0\le i \le 2N$.
Assume that
\[
g_\star \in \mu t + N_{C_t}(t)
\qquad\text{and}\qquad
-g_\star \in N_{D_t}(t).
\]
Then, with
\[
u_\star:=u_0+g_\star,
\]
we have
\[
u_\star\in \partial g(x_\star),
\qquad
-u_\star\in \partial f(x_\star).
\]
Moreover, the above inclusions on $g_\star$ are equivalent to the blockwise inequalities
\[
(g_{2k+1}-\mu t_{2k+1},\, g_{2k+2}-\mu t_{2k+2})
\cdot (t_{2k+1},t_{2k+2}) \ge 0,
\qquad k=0,\dots,N-1,
\]
\[
(g_{2k},\,g_{2k+1})\cdot (t_{2k},t_{2k+1}) \le 0,
\qquad k=0,\dots,N-1,
\]
and $g_{2N}=0$.
\end{lemma}

\begin{proof}
Recall that
\[
g(x)=\frac{\mu}{2}\|x-x_0\|^2+\langle u_0,x-x_0\rangle+\delta_{x_0+C_t}(x),
\]
and
\[
f(x)=-\langle u_0,x-x_0\rangle+\delta_{x_0+D_t}(x).
\]
Since $x_\star=x_0+t$, we have
\[
\partial g(x_\star)
=
\mu(x_\star-x_0)+u_0+N_{x_0+C_t}(x_\star)
=
\mu t+u_0+N_{C_t}(t),
\]
where we used the translation rule
\[
N_{x_0+C_t}(x_0+t)=N_{C_t}(t).
\]
Hence, if
\[
g_\star\in \mu t+ N_{C_t}(t),
\]
then
\[
u_\star:=u_0+g_\star \in \partial g(x_\star).
\]
Similarly,
\[
\partial f(x_\star)
=
-u_0+N_{x_0+D_t}(x_\star)
=
-u_0+N_{D_t}(t),
\]
and therefore
\[
-g_\star\in N_{D_t}(t)
\quad\Longrightarrow\quad
-u_\star=-(u_0+g_\star)\in \partial f(x_\star).
\]

It remains to characterize the normal cone conditions explicitly. Since $C_t$ is the Cartesian product of the singleton constraint $z_0=t_0$, the constraint $z_{-1}=0$, and the line segments
\[
[(0,0),(t_{2k-1},t_{2k})],\qquad k=1,\dots,N,
\]
its normal cone at $t$ is the product of the corresponding blockwise normal cones. The equality constraints fixing the coordinates $-1$ and $0$ contribute the full normal spaces in those coordinates, and hence impose no restrictions on the
corresponding components of a normal vector. Thus, only the two-dimensional segment blocks need to be considered. For a segment
\[
[(0,0),(a,b)],
\]
the normal cone at the endpoint $(a,b)$ is $\{w\in\mathbb{R}^2:\ w\cdot(a,b)\ge 0\}$. Applying this with
\[
w=(g_{2k+1}-\mu t_{2k+1},\, g_{2k+2}-\mu t_{2k+2})
\]
gives
\[
(g_{2k+1}-\mu t_{2k+1},\, g_{2k+2}-\mu t_{2k+2})
\cdot (t_{2k+1},t_{2k+2}) \ge 0,
\qquad k=0,\dots,N-1.
\]
Likewise, $D_t$ is the product of the segment blocks
\[
[(0,0),(t_{2k},t_{2k+1})],\qquad k=0,\dots,N-1,
\]
together with the constraint $z_{-1}=0$. Therefore,
\[
-g_\star\in N_{D_t}(t)
\]
is equivalent to
\[
(g_{2k},g_{2k+1})\cdot (t_{2k},t_{2k+1})\le 0,
\qquad k=0,\dots,N-1.
\]
Finally, $D_t$ leaves the coordinate $2N$ free and hence its normal component there must vanish. So, $g_{2N}=0$ and this proves the result. \qed
\end{proof}
We now choose the coefficients of $t$ and  $g_\star$ so that every blockwise inequality is satisfied with equality under the normalization
\[
\|x_0-x_\star\|^2+\|u_0-u_\star\|^2=1.
\] 
\begin{lemma}\label{lem:feasible_coefficients}
Define
\[
t_{2k}^\star
=
\sqrt{
\frac{\mu}{(1+2N\mu)(1+2k\mu)(1+(2k+1)\mu)}
},
\qquad k=0,\dots,N-1,
\]
\[
t_{2k+1}^\star
=
\sqrt{
\frac{\mu}{(1+2N\mu)(1+(2k+1)\mu)(1+(2k+2)\mu)}
},
\qquad k=0,\dots,N-1,
\]
\[
t_{2N}^\star=\frac{1}{1+2N\mu},
\]
and
\[
g_{2k}^\star = -(1+2k\mu)t_{2k}^\star,
\qquad
g_{2k+1}^\star = (1+(2k+2)\mu)t_{2k+1}^\star,
\qquad k=0,\dots,N-1,
\]
with
\[
g_{2N}^\star=0.
\]
Then the pair
\[
t=\sum_{i=0}^{2N}t_i^\star e_i ,
\qquad
g_\star=\sum_{i=0}^{2N}g_i^\star e_i
\]
with 
\[
x_\star = x_0+t, \quad u_\star = u_0 + g_\star
\]
satisfies
\begin{align*}
&\|x_0-x_\star\|^2 + \|u_0-u_\star\|^2 =  \sum_{i=0}^{2N} (t_i^\star)^2 + \sum_{i=0}^{2N} (g_i^\star)^2 = 1,\\
& (g^{\star}_{2k+1}-\mu t^{\star}_{2k+1},\, g^{\star}_{2k+2}-\mu t^{\star}_{2k+2})
\cdot (t^{\star}_{2k+1},t^{\star}_{2k+2}) = 0,
\quad k=0,\dots,N-1,\\
& (g^{\star}_{2k},\,g^{\star}_{2k+1})\cdot (t^{\star}_{2k},t^{\star}_{2k+1}) = 0,
\quad k=0,\dots,N-1,
\end{align*}
so that $u_\star \in \partial g(x_\star), -u_\star\in \partial f(x_\star)$, and therefore $x_\star \in \mathrm{argmin}(f+g)$.
\end{lemma}

% \[
% \inf_{y\in x_0+S_{2N-1}} \|y-x_\star\|^2 = t_{2N}^2.
% \]

\begin{proof}
We first verify the inequalities.

For each $k=0,\dots,N-2$, we have
\[
g_{2k+1}^\star-\mu t_{2k+1}^\star
=(1+(2k+1)\mu)t_{2k+1}^\star,
\]
and
\[
g_{2k+2}^\star-\mu t_{2k+2}^\star
=-(1+(2k+3)\mu)t_{2k+2}^\star.
\]
Therefore,
\begin{align*}
&(g_{2k+1}^\star-\mu t_{2k+1}^\star,\, g_{2k+2}^\star-\mu t_{2k+2}^\star)
\cdot (t_{2k+1}^\star,t_{2k+2}^\star) \\
&\qquad
=(1+(2k+1)\mu)(t_{2k+1}^\star)^2
-(1+(2k+3)\mu)(t_{2k+2}^\star)^2.
\end{align*}
By the definition of $t_{2k+1}^\star$ and $t_{2k+2}^\star$,
\[
(1+(2k+1)\mu)(t_{2k+1}^\star)^2
=
\frac{\mu}{(1+2N\mu)(1+(2k+2)\mu)},
\]
and
\[
(1+(2k+3)\mu)(t_{2k+2}^\star)^2
=
\frac{\mu}{(1+2N\mu)(1+(2k+2)\mu)}.
\]
Hence
\[
(g_{2k+1}^\star-\mu t_{2k+1}^\star,\, g_{2k+2}^\star-\mu t_{2k+2}^\star)
\cdot (t_{2k+1}^\star,t_{2k+2}^\star)=0.
\]
For $k=N-1$,
\[
g_{2N-1}^\star-\mu t_{2N-1}^\star
=
(1+(2N-1)\mu)t_{2N-1}^\star,
\qquad
g_{2N}^\star-\mu t_{2N}^\star
=
-\mu t_{2N}^\star.
\]
Therefore,
\begin{align*}
&(g_{2N-1}^\star-\mu t_{2N-1}^\star,\, g_{2N}^\star-\mu t_{2N}^\star)
\cdot (t_{2N-1}^\star,t_{2N}^\star) \\
&\qquad
=(1+(2N-1)\mu)(t_{2N-1}^\star)^2
-\mu(t_{2N}^\star)^2.
\end{align*}
Now
\[
(t_{2N-1}^\star)^2
=
\frac{\mu}{(1+2N\mu)(1+(2N-1)\mu)(1+2N\mu)}
=
\frac{\mu}{(1+(2N-1)\mu)(1+2N\mu)^2},
\]
so
\[
(1+(2N-1)\mu)(t_{2N-1}^\star)^2
=
\frac{\mu}{(1+2N\mu)^2}.
\]
Also,
\[
(t_{2N}^\star)^2=\frac{1}{(1+2N\mu)^2}.
\]
Hence
\[
(1+(2N-1)\mu)(t_{2N-1}^\star)^2
=
\mu (t_{2N}^\star)^2,
\]
and therefore
\[
(g_{2N-1}^\star-\mu t_{2N-1}^\star,\, g_{2N}^\star-\mu t_{2N}^\star)
\cdot (t_{2N-1}^\star,t_{2N}^\star)=0.
\]
Thus,
\[
(g_{2k+1}^\star-\mu t_{2k+1}^\star,\, g_{2k+2}^\star-\mu t_{2k+2}^\star)
\cdot (t_{2k+1}^\star,t_{2k+2}^\star)=0,
\qquad k=0,\dots,N-1.
\]
Similarly, for each $k=0,\dots,N-1$,
\begin{align*}
(g_{2k}^\star,g_{2k+1}^\star)\cdot (t_{2k}^\star,t_{2k+1}^\star)
&=
-(1+2k\mu)(t_{2k}^\star)^2
+(1+(2k+2)\mu)(t_{2k+1}^\star)^2.
\end{align*}
Again, by the definitions,
\[
(1+2k\mu)(t_{2k}^\star)^2
=
\frac{\mu}{(1+2N\mu)(1+(2k+1)\mu)},
\]
and
\[
(1+(2k+2)\mu)(t_{2k+1}^\star)^2
=
\frac{\mu}{(1+2N\mu)(1+(2k+1)\mu)}.
\]
Therefore,
\[
(g_{2k}^\star,g_{2k+1}^\star)\cdot (t_{2k}^\star,t_{2k+1}^\star)=0,
\qquad k=0,\dots,N-1.
\]
Thus, all block inequalities are satisfied with equality.
It remains to verify the normalization constraint. For each $j=0,\dots,2N-1$,
\[
(t_j^\star)^2
=
\frac{\mu}{(1+2N\mu)(1+j\mu)(1+(j+1)\mu)}
=
\frac{1}{1+2N\mu}
\left(
\frac{1}{1+j\mu}-\frac{1}{1+(j+1)\mu}
\right).
\]
Hence
\[
\sum_{j=0}^{2N-1}(t_j^\star)^2
=
\frac{1}{1+2N\mu}
\sum_{j=0}^{2N-1}
\left(
\frac{1}{1+j\mu}-\frac{1}{1+(j+1)\mu}
\right)
=
\frac{1}{1+2N\mu}\left(1-\frac{1}{1+2N\mu}\right).
\]
Adding $(t_{2N}^\star)^2=\frac{1}{(1+2N\mu)^2}$ yields
\[
\sum_{j=0}^{2N}(t_j^\star)^2=\frac{1}{1+2N\mu}.
\]
Next, for each $k=0,\dots,N-1$,
\[
(g_{2k}^\star)^2
=
(1+2k\mu)^2 (t_{2k}^\star)^2
=
\frac{\mu(1+2k\mu)}{(1+2N\mu)(1+(2k+1)\mu)},
\]
and
\[
(g_{2k+1}^\star)^2
=
(1+(2k+2)\mu)^2 (t_{2k+1}^\star)^2
=
\frac{\mu(1+(2k+2)\mu)}{(1+2N\mu)(1+(2k+1)\mu)}.
\]
Therefore,
\[
\begin{aligned}
(g_{2k}^\star)^2+(g_{2k+1}^\star)^2
&=
\frac{\mu}{1+2N\mu}
\left(
\frac{1+2k\mu}{1+(2k+1)\mu}
+
\frac{1+(2k+2)\mu}{1+(2k+1)\mu}
\right) \\
&=
\frac{\mu}{1+2N\mu}\cdot
\frac{2+(4k+2)\mu}{1+(2k+1)\mu}
=
\frac{2\mu}{1+2N\mu}.
\end{aligned}
\]
Summing over $k=0,\dots,N-1$ and using $g_{2N}^\star=0$, we obtain
\[
\sum_{j=0}^{2N}(g_j^\star)^2
=
\sum_{k=0}^{N-1}\big((g_{2k}^\star)^2+(g_{2k+1}^\star)^2\big)
=
\frac{2N\mu}{1+2N\mu}.
\]
Combining the two identities gives
\[
\sum_{j=0}^{2N}(t_j^\star)^2+\sum_{j=0}^{2N}(g_j^\star)^2
=
\frac{1}{1+2N\mu}+\frac{2N\mu}{1+2N\mu}
=
1.
\]
\qed
\end{proof}
Together with the fact that
\[
\|x_{N}-x_\star\|^2
\ge
\inf_{y\in x_0+S_{2N-1}}\|y-x_\star\|^2
=
(t_{2N}^\star)^2
\]
and 
\[
(t_{2N}^\star)^2
=
\frac{1}{(1+2N\mu)^2}
=
\frac{\|x_0-x_\star\|^2+\|u_0-u_\star\|^2}{1+4N^2\mu^2+4N\mu},
\]
we obtain
\[
\|x_{N}-x_\star\|^2
\ge
\frac{\|x_0-x_\star\|^2+\|u_0-u_\star\|^2}
{1+4N^2\mu^2+4N\mu}.
\]
This proves Theorem~\ref{thm:lbspan}.

\subsection{Extension using the resisting oracle technique}\label{sec:resisting}
Next, using the so-called \emph{resisting oracle technique}, we extend the result of Theorem~\ref{thm:lbspan} to all deterministic $N$-step prox-prox methods. First, we make a formal definition.

Let $N>0$ be a positive integer. A \emph{deterministic $N$-step prox-prox method} $\mathbf A$ is a mapping
from an initial pair $(x_0,u_0)$ and a tuple of functions $(f,g)$ to a sequence of outputs
\[
\mathbf A\bigl((x_0,u_0),(f,g)\bigr)
=
\bigl\{\bigl\{(z_i,\delta_i,\gamma_i)\bigr\}_{i=0}^{2N-1},\,x_{N}\big\},
\]
where $\delta_i\in\{0,1\}$ and $\gamma_i>0$ for $i=0,\dots,2N-1$. For clarity, we call $\{z_i\}_{0\le i \le 2N-1}$ the \emph{query points} and $x_{N}$ the \emph{final output}. Moreover, $\mathbf{A}$ requires that the output sequence depends on $(f,g)$ only through $(x_0,u_0)$, previous query points, and previous proximal oracle calls.
More precisely, letting
\[
y_i :=
\begin{cases}
\prox_{\gamma_i f}(z_i), & \delta_i=0,\\
\prox_{\gamma_i g}(z_i), & \delta_i=1,
\end{cases}
\qquad i=0,\dots,2N-1,
\]
there exist deterministic maps $\mathbf A_0,\dots,\mathbf A_{2N}$ such that
\[
(z_0,\delta_0,\gamma_0)=\mathbf A_0(x_0,u_0),
\]
\[
(z_i,\delta_i,\gamma_i)
=
\mathbf A_i\Bigl(
(x_0,u_0),\,
\{(z_j,\delta_j,\gamma_j,y_j)\}_{j=0}^{i-1}
\Bigr),
\qquad i=1,\dots,2N-1,
\]
and
\[
x_{N}
=
\mathbf A_{2N}\Bigl(
(x_0,u_0),\,
\{(z_j,\delta_j,\gamma_j,y_j)\}_{j=0}^{2N-1}
\Bigr).
\]

With this definition, we are now ready to formally state the more general lower bound.
\begin{theorem}\label{thm:resisting-oracle}
Fix a positive integer $N$ and let $\mu>0$. Let $d\ge 4N+4$.
For every deterministic $N$-step prox-prox method $\mathbf A$ on
$\mathbb R^d$ and every initial pair
$(x_0,u_0)\in \mathbb R^d\times \mathbb R^d$, there exist a closed,
proper, and convex function
$f\colon \mathbb R^d\to \mathbb R\cup\{\infty\}$ and a closed, proper,
and $\mu$-strongly convex function
$g\colon \mathbb R^d\to \mathbb R\cup\{\infty\}$, together with
$x_\star\in\argmin(f+g)$ and
$u_\star\in\partial g(x_\star)$ with $-u_\star\in\partial f(x_\star)$,
such that, if $x_N$ denotes the final output of
$\mathbf A$ when run from $(x_0,u_0)$ on $(f,g)$, then
\[
\|x_N-x_\star\|^2
\ge
\frac{\|x_0-x_\star\|^2+\|u_0-u_\star\|^2}
{1+4N^2\mu^2+4N\mu}.
\]
\end{theorem}

Theorem~\ref{thm:resisting-oracle}  extends Theorem~\ref{thm:lbspan} from methods satisfying the proximal span condition to all deterministic $N$-step prox-prox methods, using the resisting-oracle technique. To do this, we introduce a \emph{prox-prox zero-respecting condition}. It is closely related to the proximal span condition and may in fact be somewhat redundant here, but we include it for clarity of exposition. 

Let $f,g:\mathbb R^d\to\mathbb R\cup\{\infty\}$, and let
\[
\big\{\bigl\{(\tilde z_i,\delta_i,\gamma_i)\bigr\}_{i=0}^{2N-1},\ \tilde x_{N} \big\}
\]
be a sequence with $\delta_i\in\{0,1\}$ and $\gamma_i>0$ for
$i=0,\dots,2N-1$. Define, for each $i=0,\dots,2N-1$,
\[
\tilde d_i :=
\begin{cases}
\tilde z_i-\prox_{\gamma_i f}(\tilde z_i), & \delta_i=0,\\[1mm]
\tilde z_i-\prox_{\gamma_i g}(\tilde z_i), & \delta_i=1.
\end{cases}
\]
We say that the trajectory
\[
\big\{\bigl\{(\tilde z_i,\delta_i,\gamma_i)\bigr\}_{i=0}^{2N-1},\ \tilde x_{N} \big\}
\]
is \emph{prox-prox zero-respecting with respect to $(f,g)$} if
\[
\mathrm{supp}(\tilde z_i)\subseteq \bigcup_{k=0}^{i-1}\mathrm{supp}(\tilde d_{k}),
\qquad i=0,\dots,2N-1
\]
and
\[
\mathrm{supp}(\tilde x_{N})\subseteq  \bigcup_{k=0}^{2N-1}\mathrm{supp}(\tilde d_{k}).
\]
where the support is taken with respect to the fixed orthonormal basis with the convention $\mathrm{supp}()=\varnothing$.

The next lemma formalizes that, for functions $(f,g)$ constructed in Theorem~\ref{thm:lbspan}, the zero-respecting condition is enough to recover the proximal span condition with $(x_0, u_0)= (0,0)$. The key idea is that proximal residuals can reveal new coordinates only consecutively, one at a time. All supports below are taken with respect to the orthonormal coordinates 
$e_{-1}, e_0,\dots,e_{2N}$ used in the construction of $(f,g)$ in Section~\ref{sec:lbspan}.

\begin{lemma}\label{lem:zr-to-span}
Let $f,g:\mathbb R^{2N+2}\to\mathbb R\cup\{\infty\}$ be the functions defined in Theorem~\ref{thm:lbspan}. 
Let
\[
\big\{(\tilde z_i,\delta_i,\gamma_i)\big\}_{i=0}^{2N-1},\ \tilde x_{N} \big\}
\]
be prox-prox zero-respecting with respect to $(f,g)$, and define
\[
\tilde d_i :=
\begin{cases}
\tilde z_i-\prox_{\gamma_i f}(\tilde z_i), & \delta_i=0,\\
\tilde z_i-\prox_{\gamma_i g}(\tilde z_i), & \delta_i=1,
\end{cases}
\qquad i=0,\dots,2N-1.
\]
Then $\big\{(\tilde z_i,\delta_i,\gamma_i)\big\}_{i=0}^{2N-1},\ \tilde x_{N} \big\}$ satisfies the proximal span condition
with initial pair $(x_0,u_0)=(0,0)$, namely,
\begin{align*}
&\tilde z_i\in \mathrm{span}\{\tilde d_0,\dots,\tilde d_{i-1}\},
\qquad i=0,\dots,2N-1,\\
&\tilde x_{N}\in \mathrm{span}\{\tilde d_0,\dots,\tilde d_{2N-1}\}.
\end{align*}
Consequently, the proof of Theorem~\ref{thm:lbspan} applies to every prox-prox zero-respecting trajectory with $(x_0,u_0)=(0,0)$.
\end{lemma}

\begin{proof}
For notational simplicity, 
define
\[
E_{-1}:=\{0\},\qquad
E_m:=\operatorname{span}\{e_0,\dots,e_m\},\quad m=0,\dots,2N-1.
\]
We prove by induction on $k=0,\dots,2N-1$ that there exists
$m_k\in\{-1,0,\dots,k\}$ such that
\[
\mathrm{span}\{\tilde d_0,\dots,\tilde d_k\}=E_{m_k}.
\]
For $k=0$, the zero-respecting property gives
\[
\mathrm{supp}(\tilde z_0)\subseteq \varnothing,
\]
so $\tilde z_0=0$. Hence, by the prox-zero-chain property,
\[
\tilde d_0\in \mathrm{span}\{e_0\}.
\]
Therefore $\mathrm{span}\{\tilde d_0\}$ is either $\{0\}$ or $\mathrm{span}\{e_0\}$, so the claim holds with $m_0 = -1$ or $m_0=0$. Now suppose the claim holds for $k-1$, namely,
\[
\mathrm{span}\{\tilde d_0,\dots,\tilde d_{k-1}\}
=
E_{m_{k-1}}.
\]
Then
\[
\bigcup_{i=0}^{k-1}\mathrm{supp}(\tilde d_i)\subseteq \{0,\dots,m_{k-1}\}.
\]
Here, we use the convention $\{0,\dots,-1\}=\varnothing$ when $m_{k-1}=-1$.
Since the trajectory is prox-prox zero-respecting,
\[
\mathrm{supp}(\tilde z_k)\subseteq \bigcup_{i=0}^{k-1}\mathrm{supp}(\tilde d_i),
\]
and thus
\[
\tilde z_k\in \mathrm{span}\{e_0,\dots,e_{m_{k-1}}\}=E_{m_{k-1}}.
\]
Applying the prox-zero-chain property, we obtain
\[
\tilde d_k\in \mathrm{span}\{e_0,\dots,e_{m_{k-1}} , e_{m_{k-1}+1}\}.
\]
Therefore
\[
\mathrm{span}\{\tilde d_0,\dots,\tilde d_k\}
\]
is either
\[
\mathrm{span}\{e_0,\dots,e_{m_{k-1}}\}=E_{m_{k-1}}
\quad\text{or}\quad
\mathrm{span}\{e_0,\dots,e_{m_{k-1}+1}\}=E_{m_{k-1}+1}.
\]
Thus by induction, for each $k=0,\dots,2N-1$,
\[
\bigcup_{i=0}^{k-1}\mathrm{supp}(\tilde d_i)\subseteq \{0,\dots,m_{k-1}\},
\]
so the zero-respecting property implies
\[
\tilde z_k\in \mathrm{span}\{e_0,\dots,e_{m_{k-1}}\}
=
\mathrm{span}\{\tilde d_0,\dots,\tilde d_{k-1}\}.
\]
Likewise,
\[
\mathrm{supp}(\tilde x_{N})\subseteq \bigcup_{i=0}^{2N-1}\mathrm{supp}(\tilde d_i)
\subseteq \{0,\dots,m_{2N-1}\},
\]
and therefore
\[
\tilde x_{N}\in \mathrm{span}\{e_0,\dots,e_{m_{2N-1}}\}
=
\mathrm{span}\{\tilde d_0,\dots,\tilde d_{2N-1}\}.
\]
This proves that the trajectory satisfies the proximal span condition with initial pair $(x_0,u_0)=(0,0)$. \qed
\end{proof}

We now turn to the lifting construction used in the resisting-oracle argument.

\begin{lemma}
    Let $U \in \mathbb{R}^{d' \times d}$ have orthonormal columns with $d' \ge d$. For any vectors $x_0, u_0\in \mathbb{R}^{d'}$, if $f\colon \mathbb{R}^d \to \mathbb{R}\cup \{\infty\}$ is a closed, proper, and convex function, then $f_U: \mathbb{R}^{d'} \to \mathbb{R}\cup\{\infty\}$ defined by
    \[
    f_U(x) := f(U^\top (x-x_0)) - \langle u_0, x-x_0\rangle 
    \]
    is a convex, closed, and proper function.
\end{lemma}
\begin{proof}
Since the mapping $x\mapsto U^\top(x-x_0)$ is affine and $f$ is convex, closed, and proper,
the composition
\[
x\mapsto f\bigl(U^\top(x-x_0)\bigr)
\]
is convex, closed, and proper. Moreover, the function
\[
x\mapsto -\langle u_0,x-x_0\rangle
\]
is affine and finite-valued on $\mathbb{R}^{d'}$. Therefore,
\[
f_U(x)=f\bigl(U^\top(x-x_0)\bigr)-\langle u_0,x-x_0\rangle
\]
is the sum of a proper, closed, convex function and an affine function. Hence
$f_U$ is proper, closed, and convex. \qed
\end{proof}

\begin{lemma}
Let $U \in \mathbb{R}^{d' \times d}$ have orthonormal columns with $d' \ge d$.
For any vectors $x_0,u_0 \in \mathbb{R}^{d'}$, if $g\colon\mathbb{R}^d \to \mathbb{R}\cup\{\infty\}$ is closed, proper, and
$\mu$-strongly convex, then the function $g_U:\mathbb{R}^{d'} \to \mathbb{R}\cup\{\infty\}$
defined by
\[
g_U(x) := g\bigl(U^\top (x-x_0)\bigr) + \langle u_0, x-x_0\rangle
+ \frac{\mu}{2}\|P(x-x_0)\|^2,
\qquad
P := I-UU^\top,
\]
is $\mu$-strongly convex.
\end{lemma}

\begin{proof}
Fix $x,y\in\mathbb{R}^{d'}$ and let $v\in \partial g_U(x)$ be arbitrary and write it uniquely as
\[
v = Us+u_0 + \mu P(x-x_0)
\]
where $s\in\partial g\bigl(U^\top(x-x_0)\bigr)$.
Since $g$ is $\mu$-strongly convex, we have
\[
g\bigl(U^\top(y-x_0)\bigr)
\ge
g\bigl(U^\top(x-x_0)\bigr)
+
\left\langle s,\,U^\top(y-x)\right\rangle
+
\frac{\mu}{2}\|U^\top(y-x)\|^2.
\]
Also,
\[
\langle u_0,y-x_0\rangle
=
\langle u_0,x-x_0\rangle+\langle u_0,y-x\rangle,
\]
and
\[
\frac{\mu}{2}\|P(y-x_0)\|^2
=
\frac{\mu}{2}\|P(x-x_0)\|^2
+
\mu\langle P(x-x_0),y-x\rangle
+
\frac{\mu}{2}\|P(y-x)\|^2.
\]
Adding these identities and inequalities yields
\begin{align*}
g_U(y)
&\ge
g_U(x)
+
\langle Us+u_0+\mu P(x-x_0),\,y-x\rangle \\
&\qquad
+
\frac{\mu}{2}\|U^\top(y-x)\|^2
+
\frac{\mu}{2}\|P(y-x)\|^2.
\end{align*}
Since
\[
\|U^\top(y-x)\|^2+\|P(y-x)\|^2=\|y-x\|^2,
\]
we obtain
\begin{align*}
g_U(y)
&\ge
g_U(x)+\langle Us+u_0+\mu P(x-x_0),\,y-x\rangle+\frac{\mu}{2}\|y-x\|^2\\
&=g_U(x)+\langle v,\,y-x\rangle+\frac{\mu}{2}\|y-x\|^2.
\end{align*}
Therefore $g_U$ is $\mu$-strongly convex. \qed
\end{proof}

From now on, $U \in \mathbb{R}^{d' \times d}$ has orthonormal columns, $x_0, u_0$ are fixed, and $f_U$ and $g_U$ are defined as in the previous two lemmas. Next, we derive explicit formulas for the proximal mappings of the lifted functions.

\begin{lemma}\label{lem:fprox}
For every $z\in\mathbb{R}^{d'}$ and every $\gamma>0$, assume that $U^\top u_0=0$. Write
\[
\tilde z:=U^\top(z-x_0),
\qquad
p_z:=P(z-x_0),
\]
where $P:=I-UU^\top$. Then
\[
\mathrm{prox}_{\gamma f_U}(z)
=
x_0+U\,\mathrm{prox}_{\gamma f}(\tilde z)+p_z+\gamma u_0,
\]
and therefore
\[
z-\mathrm{prox}_{\gamma f_U}(z)
=
U\bigl(\tilde z-\mathrm{prox}_{\gamma f}(\tilde z)\bigr)-\gamma u_0.
\]
In particular,
\[
U^\top\!\left(z-\mathrm{prox}_{\gamma f_U}(z)\right)
=
\tilde z-\mathrm{prox}_{\gamma f}(\tilde z).
\]
\end{lemma}

\begin{proof}
Since $U^\top U=I_d$ and $P=I-UU^\top$, every $x\in\mathbb{R}^{d'}$ admits a unique decomposition
\[
x=x_0+U\tilde x+p,
\qquad
\tilde x\in\mathbb{R}^d,\quad p\in\mathrm{range}(P),
\]
and we write
\[
z=x_0+U\tilde z+p_z,
\qquad
\tilde z:=U^\top(z-x_0),\quad p_z:=P(z-x_0).
\]
Since $U^\top u_0=0$, we have $u_0\in\mathrm{range}(P)$, and hence
\[
\langle u_0,x-x_0\rangle
=
\langle u_0,U\tilde x+p\rangle
=
\langle u_0,p\rangle.
\]
Using $\mathrm{range}(U)\perp\mathrm{range}(P)$, we obtain
\[
\|x-z\|^2
=
\|\tilde x-\tilde z\|^2+\|p-p_z\|^2,
\]
and therefore
\[
f_U(x)+\frac{1}{2\gamma}\|x-z\|^2
=
f(\tilde x)
-\langle u_0,p\rangle
+\frac{1}{2\gamma}\|\tilde x-\tilde z\|^2
+\frac{1}{2\gamma}\|p-p_z\|^2.
\]
Hence the minimization separates. Minimizing with respect to $\tilde x$, we get
\[
\tilde x=\mathrm{prox}_{\gamma f}(\tilde z),
\]
while the minimizing $p$ satisfies
\[
\frac{1}{\gamma}(p-p_z)-u_0=0,
\]
so
\[
p=p_z+\gamma u_0.
\]
This proves that
\[
\mathrm{prox}_{\gamma f_U}(z)
=
x_0+U\,\mathrm{prox}_{\gamma f}(\tilde z)+p_z+\gamma u_0.
\]
Subtracting from $z=x_0+U\tilde z+p_z$ gives
\[
z-\mathrm{prox}_{\gamma f_U}(z)
=
U\bigl(\tilde z-\mathrm{prox}_{\gamma f}(\tilde z)\bigr)-\gamma u_0.
\]
Finally, since $U^\top u_0=0$, applying $U^\top$ yields
\[
U^\top\!\left(z-\mathrm{prox}_{\gamma f_U}(z)\right)
=
\tilde z-\mathrm{prox}_{\gamma f}(\tilde z).
\]
\qed
\end{proof}

\begin{lemma}\label{lem:gprox}
For every $z\in\mathbb{R}^{d'}$ and every $\gamma>0$, assume that $U^\top u_0=0$. Write
\[
\tilde z:=U^\top(z-x_0),
\qquad
p_z:=P(z-x_0),
\]
where $P:=I-UU^\top$. Then
\[
\mathrm{prox}_{\gamma g_U}(z)
=
x_0+U\,\mathrm{prox}_{\gamma g}(\tilde z)+\frac{1}{1+\gamma\mu}(p_z-\gamma u_0),
\]
and
\[
z-\mathrm{prox}_{\gamma g_U}(z)
=
U\bigl(\tilde z-\mathrm{prox}_{\gamma g}(\tilde z)\bigr)
+
\frac{\gamma\mu}{1+\gamma\mu}\,p_z
+
\frac{\gamma}{1+\gamma\mu}\,u_0.
\]
In particular,
\[
U^\top\!\left(z-\mathrm{prox}_{\gamma g_U}(z)\right)
=
\tilde z-\mathrm{prox}_{\gamma g}(\tilde z).
\]
\end{lemma}

\begin{proof}
Again write
\[
x=x_0+U\tilde x+p,
\qquad
z=x_0+U\tilde z+p_z,
\qquad
p,p_z\in\mathrm{range}(P).
\]
Since $U^\top u_0=0$, we have $u_0\in\mathrm{range}(P)$, and therefore
\[
\langle u_0,x-x_0\rangle
=
\langle u_0,U\tilde x+p\rangle
=
\langle u_0,p\rangle.
\]
Using $\mathrm{range}(U)\perp \mathrm{range}(P)$, we obtain
\[
\|x-z\|^2
=
\|\tilde x-\tilde z\|^2+\|p-p_z\|^2.
\]
Hence
\begin{align*}
g_U(x)+\frac{1}{2\gamma}\|x-z\|^2
&=
g(\tilde x)
+
\langle u_0,p\rangle
+
\frac{\mu}{2}\|p\|^2
+
\frac{1}{2\gamma}\|\tilde x-\tilde z\|^2
+
\frac{1}{2\gamma}\|p-p_z\|^2.
\end{align*}
The minimization separates. Minimizing with respect to $\tilde x$, we get
\[
\tilde x=\mathrm{prox}_{\gamma g}(\tilde z),
\]
while the minimizing $p$ satisfies
\[
u_0+\mu p+\frac{1}{\gamma}(p-p_z)=0.
\]
Equivalently,
\[
(1+\gamma\mu)p=p_z-\gamma u_0,
\]
so
\[
p=\frac{1}{1+\gamma\mu}(p_z-\gamma u_0).
\]
This proves that
\[
\mathrm{prox}_{\gamma g_U}(z)
=
x_0+U\,\mathrm{prox}_{\gamma g}(\tilde z)+\frac{1}{1+\gamma\mu}(p_z-\gamma u_0).
\]
Subtracting from
\[
z=x_0+U\tilde z+p_z
\]
gives
\begin{align*}
z-\mathrm{prox}_{\gamma g_U}(z)
&=
U\bigl(\tilde z-\mathrm{prox}_{\gamma g}(\tilde z)\bigr)
+
p_z-\frac{1}{1+\gamma\mu}(p_z-\gamma u_0) \\
&=
U\bigl(\tilde z-\mathrm{prox}_{\gamma g}(\tilde z)\bigr)
+
\frac{\gamma\mu}{1+\gamma\mu}\,p_z
+
\frac{\gamma}{1+\gamma\mu}\,u_0.
\end{align*}
Finally, since $U^\top u_0=0$ and $U^\top p_z=0$, applying $U^\top$ yields
\[
U^\top\!\left(z-\mathrm{prox}_{\gamma g_U}(z)\right)
=
\tilde z-\mathrm{prox}_{\gamma g}(\tilde z).  
\]\qed
\end{proof}

These identities allow us to match proximal evaluations in the lifted space with projected proximal evaluations in the original space. We are now ready to carry out the resisting-oracle construction.

\begin{lemma}\label{lem:resisting}
Let $N>0$ be a positive integer. Assume $\mu>0$, $d'\ge d+2N+2$, and let $\textbf{A}$ be any deterministic
$N$-step prox-prox method with initial pair
$(x_0,u_0)\in\mathbb{R}^{d'}\times\mathbb{R}^{d'}$. Let
\[
f\colon\mathbb{R}^d\to\mathbb{R}\cup\{\infty\}
\quad\text{and}\quad
g\colon\mathbb{R}^d\to\mathbb{R}\cup\{\infty\}
\]
be such that $f$ is closed, proper, and convex and $g$ is
closed, proper, and $\mu$-strongly convex. Then there exists a matrix
\[
U\in\mathbb{R}^{d'\times d}
\]
with orthonormal columns such that $U^\top u_0=0$, and if
\[
\textbf{A}\bigl((x_0,u_0),(f_U,g_U)\bigr)
=
\bigl\{\bigl\{(z_i,\delta_i,\gamma_i)\bigr\}_{i=0}^{2N-1},\,x_{N}\big\},
\]
and we define the projected iterates
\[
\tilde z_i:=U^\top(z_i-x_0),
\qquad
\tilde x_{N}:=U^\top(x_{N}-x_0),
\]
then the sequence
\[
\big\{\bigl\{(\tilde z_i,\delta_i,\gamma_i)\bigr\}_{i=0}^{2N-1},\,\tilde x_{N}\big\}
\]
is prox-prox zero-respecting with respect to $(f,g)$.
\end{lemma}

\begin{proof}
Let $e_0,\dots,e_{d-1}$ be the standard basis of $\mathbb R^d$.
First, we recursively construct
\[
I_0\subseteq I_1\subseteq\cdots\subseteq I_{2N}\subseteq\{0,\dots,d-1\},
\]
an orthonormal family $\{v_j\}_{j\in I_{2N}}\subseteq\mathbb R^{d'}$, and a method trajectory
\[
(z_0,\delta_0,\gamma_0,y_0),\dots,(z_{2N-1},\delta_{2N-1},\gamma_{2N-1},y_{2N-1}),
\]
where $y_i\in\mathbb R^{d'}$ is intended to be the output of the $i$-th prox call.

We begin with
\[
I_0:=\varnothing.
\]
Next, fix $i\in\{0,\dots,2N-1\}$, and suppose that $I_i$, the vectors
$\{v_j\}_{j\in I_i}$, and the previous trajectory
\[
(z_0,\delta_0,\gamma_0,y_0),\dots,(z_{i-1},\delta_{i-1},\gamma_{i-1},y_{i-1})
\]
have already been constructed. Since $\mathbf A$ is deterministic, the trajectory
determines the next query triple $(z_i,\delta_i,\gamma_i)$. Define
\[
\bar z_i:=\sum_{j\in I_i}\langle v_j,z_i-x_0\rangle e_j\in\mathbb R^d,
\]
so that $\mathrm{supp}(\bar z_i)\subseteq I_i$. Next define
\[
\bar y_i:=
\begin{cases}
\mathrm{prox}_{\gamma_i f}(\bar z_i), & \delta_i=0,\\[1mm]
\mathrm{prox}_{\gamma_i g}(\bar z_i), & \delta_i=1,
\end{cases}
\]
and the corresponding residual
\[
\bar d_i:=\bar z_i-\bar y_i.
\]
Set
\[
I_{i+1}:=I_i\cup \mathrm{supp}(\bar d_i).
\]
Since $\bar y_i=\bar z_i-\bar d_i$, we also have
\[
\mathrm{supp}(\bar y_i)\subseteq I_{i+1}.
\]
We now choose the new vectors
\[
\{v_j\}_{j\in I_{i+1}\setminus I_i}
\]
as an orthonormal family in $S_i^\perp$, where
\[
S_i:=\operatorname{span}\Bigl(\{u_0,z_0-x_0,\dots,z_i-x_0\}\cup\{v_j\}_{j\in I_i}\Bigr).
\]
This is possible because
\[
\dim S_i\le (i+2)+|I_i|,
\]
hence
\[
\dim S_i^\perp
\ge d'-(i+2)-|I_i|
\ge d-|I_i|
\ge |I_{i+1}\setminus I_i|.
\]
Now define
\[
q_i:=(z_i-x_0)-\sum_{j\in I_i}\langle v_j,z_i-x_0\rangle v_j,
\]
and $y_i\in\mathbb R^{d'}$ by
\[
y_i:=
\begin{cases}
\displaystyle
x_0+\sum_{j\in I_{i+1}}(\bar y_i)_j\,v_j+q_i+\gamma_i u_0,
& \delta_i=0,\\[2mm]
\displaystyle
x_0+\sum_{j\in I_{i+1}}(\bar y_i)_j\,v_j+\frac{1}{1+\gamma_i\mu}(q_i-\gamma_i u_0),
& \delta_i=1.
\end{cases}
\]
This completes the recursive construction for $i=0,\dots,2N-1$. Since $\mathbf A$ is deterministic, this transcript uniquely determines the final output $x_{N}\in\mathbb R^{d'}$. Then choose the orthonormal vectors indexed by $j\notin I_{2N}$ in the orthonormal complement of
\[
S_{2N+1}:=\operatorname{span}\Bigl(\{u_0,z_0-x_0,\dots,z_{2N-1}-x_0,\ x_{N}-x_0\}\cup\{v_j\}_{j\in I_{2N}}\Bigr).
\]
This is possible because
\[
\dim S_{2N+1}\le (2N+2)+|I_{2N}|,
\]
so
\[
\dim S_{2N+1}^\perp\ge d'-(2N+2)-|I_{2N}|\ge d-|I_{2N}|.
\]
Now set
\[
U:=[v_0|\cdots|v_{d-1}]\in\mathbb R^{d'\times d},
\qquad
P:=I-UU^\top.
\]
By construction, each column of $U$ is orthogonal to $u_0$, and therefore $U^\top u_0=0$.

For the next step, we claim that if $\mathbf A$ is run on $(f_U,g_U)$, then after $i$ oracle calls its trajectory is exactly
\[
(z_0,\delta_0,\gamma_0,y_0),\dots,(z_{i-1},\delta_{i-1},\gamma_{i-1},y_{i-1}),
\qquad i=0,\dots,2N,
\]
and consequently, its final output is $x_{N}$. We prove this by induction on $i$. For $i=0$, we only have the initial state $(x_0,u_0)$ and empty transcript, so the claim is immediate. Now assume the claim holds for some $i\in\{0,\dots,2N-1\}$.
Then $\mathbf A$ run on $(f_U,g_U)$ has seen exactly the same trajectory
\[
(z_0,\delta_0,\gamma_0,y_0),\dots,(z_{i-1},\delta_{i-1},\gamma_{i-1},y_{i-1})
\]
as in the recursive construction. Hence, its next query is exactly $(z_i,\delta_i,\gamma_i)$.

Now for every $j\notin I_{i}$, the vector $v_j$ is orthogonal to $z_i-x_0$, so $\langle v_j,z_i-x_0\rangle=0$.
Therefore,
\[
U^\top(z_i-x_0)=\bar z_i,
\qquad
P(z_i-x_0)=q_i.
\]
Now we apply Lemma~\ref{lem:fprox} and Lemma~\ref{lem:gprox}. If $\delta_i=0$, then
\begin{align*}
\mathrm{prox}_{\gamma_i f_U}(z_i)
&=
x_0+U\,\mathrm{prox}_{\gamma_i f}\bigl(U^\top(z_i-x_0)\bigr)+P(z_i-x_0)+\gamma_i u_0\\
&=
x_0+U\,\mathrm{prox}_{\gamma_i f}(\bar z_i)+q_i+\gamma_i u_0.
\end{align*}
Since $\bar y_i=\mathrm{prox}_{\gamma_i f}(\bar z_i)$ and
$\mathrm{supp}(\bar y_i)\subseteq I_{i+1}$, this gives
\[
\mathrm{prox}_{\gamma_i f_U}(z_i)
=
x_0+\sum_{j\in I_{i+1}}(\bar y_i)_j\,v_j+q_i+\gamma_i u_0
=
y_i.
\]
If $\delta_i=1$, then
\begin{align*}
\mathrm{prox}_{\gamma_i g_U}(z_i)
&=
x_0+U\,\mathrm{prox}_{\gamma_i g}\!\bigl(U^\top(z_i-x_0)\bigr)
+\frac{1}{1+\gamma_i\mu}\bigl(P(z_i-x_0)-\gamma_i u_0\bigr)\\
&=
x_0+U\,\mathrm{prox}_{\gamma_i g}(\bar z_i)
+\frac{1}{1+\gamma_i\mu}(q_i-\gamma_i u_0).
\end{align*}
Since $\bar y_i=\mathrm{prox}_{\gamma_i g}(\bar z_i)$ and
$\mathrm{supp}(\bar y_i)\subseteq I_{i+1}$, this gives
\[
\mathrm{prox}_{\gamma_i g_U}(z_i)
=
x_0+\sum_{j\in I_{i+1}}(\bar y_i)_j\,v_j+\frac{1}{1+\gamma_i\mu}(q_i-\gamma_i u_0)
=
y_i.
\]

Thus the true $i$-th oracle output is exactly $y_i$, which proves the induction step.
Hence the actual run of $\mathbf A$ on $(f_U,g_U)$ has exactly the recursively constructed $2N$-step transcript. Since the final output depends only on this trajectory, the final output is exactly $x_{N}$.

We now verify the zero-respecting property.
For each $i=0,\dots,2N-1$, define
\[
\tilde d_i:=
\begin{cases}
\tilde z_i-\mathrm{prox}_{\gamma_i f}(\tilde z_i), & \delta_i=0,\\
\tilde z_i-\mathrm{prox}_{\gamma_i g}(\tilde z_i), & \delta_i=1.
\end{cases}
\]
Since $y_i=\mathrm{prox}_{\gamma_i f_U}(z_i)$ when $\delta_i=0$ and
$y_i=\mathrm{prox}_{\gamma_i g_U}(z_i)$ when $\delta_i=1$, Lemma~\ref{lem:fprox}
and Lemma~\ref{lem:gprox} give
\[
\tilde d_i = U^\top(z_i-y_i).
\]
Next, we show that
\[
U^\top(y_i-x_0)=\bar y_i.
\]
If $\delta_i=0$, then by construction,
\[
y_i=
x_0+\sum_{j\in I_{i+1}}(\bar y_i)_j\,v_j+q_i+\gamma_i u_0.
\]
Multiplying by $U^\top$, and using $U^\top q_i=0$ and $U^\top u_0=0$, we obtain
\[
U^\top(y_i-x_0)=\bar y_i.
\]
If $\delta_i=1$, then
\[
y_i=
x_0+\sum_{j\in I_{i+1}}(\bar y_i)_j\,v_j+\frac{1}{1+\gamma_i\mu}(q_i-\gamma_i u_0),
\]
and again $U^\top q_i=0$ and $U^\top u_0=0$ imply
\[
U^\top(y_i-x_0)=\bar y_i.
\]
Since also
\[
U^\top(z_i-x_0)=\bar z_i,
\]
we conclude that
\[
\tilde d_i
=
U^\top(z_i-y_i)
=
U^\top(z_i-x_0)-U^\top(y_i-x_0)
=
\bar z_i-\bar y_i
=
\bar d_i.
\]
Because $I_0=\varnothing$ and
\[
I_{i+1}=I_i\cup\mathrm{supp}(\bar d_i)
      =I_i\cup\mathrm{supp}(\tilde d_i),
\]
it follows by induction that
\[
I_i=\bigcup_{k=0}^{i-1}\mathrm{supp}(\tilde d_k),
\qquad i=0,\dots,2N.
\]
Moreover,
\[
\mathrm{supp}(\tilde z_i)
=
\mathrm{supp}(\bar z_i)
\subseteq I_i
=
\bigcup_{k=0}^{i-1}\mathrm{supp}(\tilde d_k),
\qquad i=0,\dots,2N-1.
\]
Finally, for every $j\notin I_{2N}$ we have $v_j\in S_{2N+1}^\perp$ and
$x_{N}-x_0\in S_{2N+1}$, so
\[
\langle v_j,x_{N}-x_0\rangle=0.
\]
Therefore
\[
\mathrm{supp}(\tilde x_{N})
=
\mathrm{supp}\bigl(U^\top(x_{N}-x_0)\bigr)
\subseteq I_{2N}
=
\bigcup_{k=0}^{2N-1}\mathrm{supp}(\tilde d_k).
\]
Hence
\[
\mathrm{supp}(\tilde x_{N})
\subseteq
\bigcup_{k=0}^{2N-1}\mathrm{supp}(\tilde d_k),
\]
which is exactly the prox-prox zero-respecting property. \qed
\end{proof}

Combining the resisting-oracle lemma with the zero-respecting property, we obtain the desired lower bound for arbitrary deterministic prox-prox methods.

\begin{proof}
By Theorem~\ref{thm:lbspan} with initial pair $(0,0)$, there exist a closed, proper, and convex function
$\tilde f\colon \mathbb R^{2N+2}\to \mathbb R\cup\{\infty\}$ and a closed, proper, and $\mu$-strongly convex function
$\tilde g\colon \mathbb R^{2N+2}\to \mathbb R \cup\{\infty\}$
with
\[
\tilde x_\star\in \mathrm{argmin}(\tilde f+\tilde g),
\qquad
\tilde u_\star\in \partial \tilde g(\tilde x_\star), \, -\tilde u_\star \in \partial \tilde f(\tilde x_\star).
\]
Since $d\ge 4N+4=(2N+2)+(2N+2)$, by Lemma~\ref{lem:resisting} there exists a matrix
\[
U\in \mathbb R^{d\times(2N+2)}
\]
with orthonormal columns such that $U^\top u_0=0$, and if
\[
\mathbf A\bigl((x_0,u_0),(\tilde f_U,\tilde g_U)\bigr)
=
\big\{\bigl\{(z_i,\delta_i,\gamma_i)\bigr\}_{i=0}^{2N-1},\,x_{N} \big\},
\]
and we define
\[
\tilde z_i:=U^\top(z_i-x_0),
\qquad
\tilde x_{N}:=U^\top(x_{N}-x_0),
\]
then the projected sequence
\[
\big\{\bigl\{(\tilde z_i,\delta_i,\gamma_i)\bigr\}_{i=0}^{2N-1},\,\tilde x_{N} \big\}
\]
is prox-prox zero-respecting with respect to $(\tilde f,\tilde g)$. Therefore, by Theorem~\ref{thm:lbspan} and Lemma~\ref{lem:zr-to-span},
\begin{equation}\label{eq:lbspan}
\|\tilde x_{N}-\tilde x_\star\|^2
\ge
\frac{\|\tilde x_\star\|^2 + \|\tilde u_\star\|^2}{1+4N^2\mu^2+4N\mu}.
\end{equation}
Now define
\[
f:=\tilde f_U,
\qquad
g:=\tilde g_U,
\]
and
\[
x_\star:=x_0+U\tilde x_\star,
\qquad
u_\star:=u_0+U\tilde u_\star.
\]
We first show that $x_\star\in \mathrm{argmin}(f+g)$.
Every $z\in\mathbb R^d$ can be written uniquely as
\[
z=x_0+Uy+p,
\qquad
y\in\mathbb R^{2N+2},\quad p\in \mathrm{range}(P),
\]
where $P:=I-UU^\top$. By definition of $f_U$ and $g_U$, we have
\[
f(z)=\tilde f(y)-\langle u_0,Uy+p\rangle,
\]
and
\[
g(z)=\tilde g(y)+\langle u_0,Uy+p\rangle+\frac{\mu}{2}\|p\|^2.
\]
Hence
\[
(f+g)(z)=\tilde f(y)+\tilde g(y)+\frac{\mu}{2}\|p\|^2.
\]
This is minimized exactly when
\[
y\in \mathrm{argmin}(\tilde f+\tilde g)
\quad\text{and}\quad
p=0.
\]
Thus
\[
x_\star=x_0+U\tilde x_\star\in \mathrm{argmin}(f+g).
\]
Next we show that $u_\star\in \partial g(x_\star)$ and $-u_\star\in \partial f(x_\star)$.
By definition,
\[
g(x)=\tilde g\bigl(U^\top(x-x_0)\bigr)+\langle u_0,x-x_0\rangle+\frac{\mu}{2}\|P(x-x_0)\|^2.
\]
Hence
\[
\partial g(x)
=
U\,\partial \tilde g\bigl(U^\top(x-x_0)\bigr)+u_0+\mu P(x-x_0).
\]
Since
\[
x_\star=x_0+U\tilde x_\star,
\]
we have
\[
U^\top(x_\star-x_0)=\tilde x_\star,
\qquad
P(x_\star-x_0)=0.
\]
Therefore
\[
\partial g(x_\star)=U\,\partial \tilde g(\tilde x_\star)+u_0.
\]
Since $\tilde u_\star\in \partial \tilde g(\tilde x_\star)$, we conclude that
\[
u_\star=u_0+U\tilde u_\star\in \partial g(x_\star).
\]
Similarly, 
\[
f(x)=\tilde f\bigl(U^\top(x-x_0)\bigr)-\langle u_0,x-x_0\rangle.
\]
So, 
\[
\partial f(x) = U\partial \tilde{f}\bigl(U^\top(x-x_0)\bigr)-u_0,
\]
and
\[
\partial f(x_\star) = U\partial \tilde{f}(\tilde x_\star) - u_0.
\]
Thus, we conclude that $-u_\star= -u_0 -U\tilde u_\star \in \partial f(x_\star)$. Putting altogether, we have
\[
\|x_0-x_\star\|^2=\|U\tilde x_\star\|^2=\|\tilde x_\star\|^2,
\]
and
\[
\|u_0-u_\star\|^2=\|U\tilde u_\star\|^2=\|\tilde u_\star\|^2.
\]
Thus
\[
\|x_0-x_\star\|^2+\|u_0-u_\star\|^2
=
\|\tilde x_\star\|^2+\|\tilde u_\star\|^2.
\]
Finally,
\[
x_{N}-x_\star
=
x_{N}-x_0-U\tilde x_\star
=
U(\tilde x_{N}-\tilde x_\star)+P(x_{N}-x_0),
\]
and the two terms on the right-hand side are orthogonal. Hence
\[
\|x_{N}-x_\star\|^2
=
\|\tilde x_{N}-\tilde x_\star\|^2+\|P(x_{N}-x_0)\|^2
\ge
\|\tilde x_{N}-\tilde x_\star\|^2.
\]
Combining with \eqref{eq:lbspan} yields
\[
\|x_{N}-x_\star\|^2
\ge
\frac{\|x_0-x_\star\|^2+\|u_0-u_\star\|^2}{1+4N^2\mu^2+4N\mu}.
\]
\qed
\end{proof}

\section{Numerical Results}
We compare FDR with FISTA, Douglas--Rachford splitting (DRS), accelerated
Chambolle--Pock (CP), and accelerated Davis--Yin splitting (DYS) on a family
of elastic-net regression problems.  Each instance has the form
\begin{equation*}
    \begin{split}
        \underset{x}{\text{minimize}} 
        \quad & \underbrace{\|Ax - b\|_2^2 + \frac{\mu}{2}\|x\|^2}_{=g(x)} + \underbrace{\lambda \|x\|_1}_{=f(x)} .  \\
    \end{split}
\end{equation*}
We randomly generated $100$ instances with $A \in \mathbb{R}^{40\times
100}$ and $b \in \mathbb{R}^{40}$ being computed as $Ax_{\mathrm{true}} + \epsilon$, where $\epsilon$ is drawn from a Gaussian with standard deviation $0.01$, and $x_{\mathrm{true}}$ is generated randomly with only $10\%$ of elements being nonzero. We use $\mu=\lambda=10^{-3}$.

Since the guarantee on FDR's squared distance to the solution is on the final iterate, $x_N$, the FDR curve in Figure~\ref{fig:elastic-net-regression} reports the terminal error \(\|x_k-x_\star\|^2\) obtained by running FDR for $N=k$ iterations. The solid curves are the median of each algorithm's performance, and the shaded region shows the interquartile range for each on the $100$ instances.

\begin{figure}[t]
    \centering
    \includegraphics[width=0.75\textwidth]{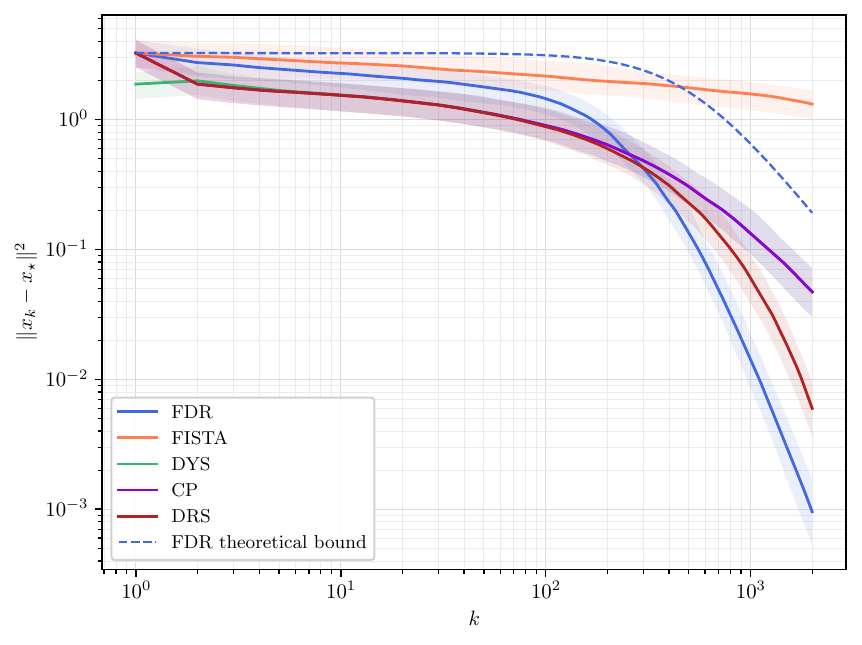}
    \caption{Median squared distance to the solution over \(100\) elastic-net
    instances.  Shaded bands show the interquartile range.}
    \label{fig:elastic-net-regression}
\end{figure}

Figure~\ref{fig:elastic-net-regression} shows that FDR eventually achieves the smallest error among the tested methods. The observed FDR errors also remain below the theoretical bound, which is valid, but we see these errors are much smaller than the theoretical upper bound on these random elastic-net instances.

\FloatBarrier
\section{Conclusion}
In this work, we presented Fast Douglas--Rachford Splitting (FDR), an accelerated Douglas--Rachford-type method for minimizing the sum of a convex and a $\mu$-strongly convex function. We established the guarantee
\[
    \|x_N-x_\star\|^2
    \le
    \frac{\|x_0-x_\star\|^2+\|u_0-u_\star\|^2}{1+4N^2\mu^2},
\]
which improves upon the leading asymptotic constant of the accelerated methods of Chambolle--Pock and Davis--Yin by a factor of $4$. We further established a complexity lower bound showing that both the $\mathcal{O}(1/N^2)$ convergence rate and the leading-order constant achieved by FDR are optimal.

It is worth noting that the acceleration mechanism and the lower-bound construction are fundamentally different than those of Nesterov acceleration. In this sense, the accelerated methods of Chambolle--Pock, Davis--Yin, and FDR should be viewed as belonging to a class of acceleration mechanism distinct from the Nesterov-type acceleration. Developing a more principled understanding of the different classes of acceleration mechanisms is an interesting direction for future work.

Finally, while our lower bound asymptotically matches the upper bound in the leading-order term, a gap remains in the higher-order terms. Since FDR was discovered through BnB-PEP \cite{das2024branch}, a framework for discovering optimal first-order algorithms, we conjecture that the upper bound achieved by FDR is in fact exactly optimal, implying that the current lower bound is likely loose. We leave the task of fully tightening the lower bound to future work.

\begin{acknowledgements}
GMC and BA were supported by the Office of Naval Research grants N000142512231 and N00014-25-1-2319. UJ and EKR were supported by the Air Force Office of Scientific Research under award number FA95502510183.
\end{acknowledgements}

% \section*{Declarations}

\begin{appendices}

\section{Leading constants for accelerated splitting methods}\label{apx:leading_constants}

In this appendix we provide the asymptotic leading constants of the accelerated Davis--Yin splitting algorithm and the accelerated Chambolle--Pock algorithm when specialized to \eqref{eq:primal}. 

\subsection{Accelerated Davis--Yin Splitting}

From Proposition A.1 in \cite{davis2017three}, the iterates of Algorithm \eqref{eq:DYS} satisfy
\[
(1+2\gamma_k\mu)\|x_{k+1}-x_\star\|^2 + \gamma_k^2\|u_{k+1}-u_\star\|^2
\le
\|x_k-x_\star\|^2 + \gamma_k^2\|u_k-u_\star\|^2.
\]

Furthermore, in the proof of Theorem 1.2 in Appendix A of \cite{davis2017three}, we see that the proximal step sizes satisfy
\[
\frac{1}{\gamma_{k+1}^2} = \frac{1+2\gamma_k\mu}{\gamma_k^2},
\]
and
\[
\lim_{k\to\infty} (k+1)\gamma_k = \frac{1}{\mu}.
\]

After $N$ proximal evaluations of $g$ we obtain
\begin{align*}
\|x_{N-1}-x_\star\|^2
&\le
\gamma_{N-1}^2
\left(
\frac{1}{\gamma_0^2}\|x_0-x_\star\|^2 + \|u_0-u_\star\|^2
\right) \\
&=
\frac{\gamma_{N-1}^2}{\gamma_0^2}
\left(
\|x_0-x_\star\|^2 + \gamma_0^2\|u_0-u_\star\|^2
\right).
\end{align*}

Setting $\gamma_0^2 = 1$ to match our assumed initial condition yields the asymptotic rate
\[
\|x_{N-1}-x_\star\|^2 \sim \frac{\|x_0-x_\star\|^2 + \|u_0-u_\star\|^2}{N^2\mu^2}.
\]

\subsection{Accelerated Chambolle--Pock}

The convergence rate of the accelerated Chambolle--Pock algorithm is given in \cite[Theorem 2]{Chambolle2010}, assuming $\sigma_0=1/\tau_0$. For any $\epsilon>0$, there exists $N_0$ such that for all $N\ge N_0$,
\begin{align*}
\|x_N-x_\star\|^2
&\le
\frac{1+\epsilon}{N^2}
\left(
\frac{\|x_0-x_\star\|^2}{\mu^2\tau_0^2}
+
\frac{\|u_0-u_\star\|^2}{\mu^2}
\right) \\
&=
\frac{1+\epsilon}{N^2\mu^2\tau_0^2}
\left(
\|x_0-x_\star\|^2 + \tau_0^2\|u_0-u_\star\|^2
\right).
\end{align*}

Setting $\tau_0^2 = 1$ to match our assumed initial condition and taking $\epsilon \to 0$ yields
\[
\|x_N-x_\star\|^2 \sim \frac{\|x_0-x_\star\|^2 + \|u_0-u_\star\|^2}{N^2\mu^2}.
\]

\subsection{Comparison with FDR}

Recall that FDR satisfies
\begin{align*}
\|x_N-x_\star\|^2
&\le
\frac{\|x_0-x_\star\|^2 + \|u_0-u_\star\|^2}{1+4N^2\mu^2}
\sim
\frac{\|x_0-x_\star\|^2 + \|u_0-u_\star\|^2}{4N^2\mu^2}.
\end{align*}

Thus FDR improves the asymptotic constant of both accelerated Davis--Yin splitting and accelerated Chambolle--Pock by a factor of $4$.

\end{appendices}

% Authors must disclose all relationships or interests that 
% could have direct or potential influence or impart bias on 
% the work: 
%
% \section*{Conflict of interest}
%
% The authors declare that they have no conflict of interest.

\bibliographystyle{spmpsci}
\bibliography{ref.bib}

\end{document}